\documentclass{article}
\usepackage[a4paper, margin=25mm]{geometry}
\usepackage{amsmath,amssymb}
\usepackage{eucal}

\usepackage[raiselinks=false,colorlinks=true,citecolor=blue,urlcolor=blue,linkcolor=blue,bookmarksopen=true]{hyperref}

\usepackage{manfnt}
\usepackage{tikz-cd}
\usetikzlibrary{arrows}

\tikzset{
    onto/.style={/tikz/commutative diagrams/twoheadrightarrow}
}

\usepackage[hang,flushmargin]{footmisc}

\newcommand\blfootnote[1]{%
  \begingroup
  \renewcommand\thefootnote{}\footnote{#1}%
  \addtocounter{footnote}{-1}%
  \endgroup
}

\DeclareSymbolFont{bbold}{U}{bbold}{m}{n}
\DeclareSymbolFontAlphabet{\mathbbold}{bbold}

\newcommand{\one}{\mathbbold{1}}

\newcommand{\commutes}{\texttt{"}}

\newcommand{\strict}{\texttt{\footnotesize str-pbk}}
\newcommand{\ho}{\texttt{\footnotesize ho-pbk}}

\newcommand{\ground}{k}
\newcommand{\pari}{\operatorname{par}}

\newcommand{\e}{\mathsf{e}}

\newcommand{\lin}{\kat{lin}}
\newcommand{\linpm}{\kat{lin}^\pm}

\newcommand{\FF}{\mathcal{F}}
\newcommand{\ZZ}{\mathcal{Z}}

\newcommand{\B}{\mathbb{B}}
\newcommand{\F}{\mathbb{F}}
\renewcommand{\S}{\mathbb{S}}
\newcommand{\N}{\mathbb{N}}
\newcommand{\Z}{\mathbb{Z}}
\newcommand{\Q}{\mathbb{Q}}

\newcommand{\CC}{\mathcal{C}}
\newcommand{\DD}{\mathcal{D}}

\newcommand{\isopil}{\stackrel{\raisebox{0.1ex}[0ex][0ex]{\(\sim\)}}%
			{\raisebox{-0.15ex}[0.28ex]{\(\rightarrow\)}}}

\providecommand{\kat}[1]{\text{\textbf{\textsl{#1\/}}}}

\newcommand{\vect}{\kat{vect}}
\providecommand{\basis}{\textbf{u}}
\newcommand{\name}[1]{\ulcorner #1\urcorner}

\newcommand{\pp}{{\raisebox{1.2pt}{\text{\tiny \(\oplus\)}}}}
\newcommand{\mm}{{\raisebox{1.2pt}{\text{\tiny \(\ominus\)}}}}

\renewcommand{\P}{\mathbb{P}}

\newcommand{\tensor}{\otimes}
\newcommand{\atensor}{\circledast}

\newcommand{\Fun}{\operatorname{Fun}}
\newcommand{\Aut}{\operatorname{Aut}}

\newcommand{\id}{\operatorname{id}}

\DeclareRobustCommand\upperstar{%
  \mathchoice%
    {\kern0pt\raise0.55ex\hbox{$\displaystyle *$}\kern0.8pt}
    {\kern0pt\raise0.58ex\hbox{$\textstyle *$}\kern0.8pt}
    {\kern0pt\raise0.45ex\hbox{$\scriptstyle *$}\kern0.4pt}
    {\kern0pt\raise0.4ex\hbox{$\scriptscriptstyle *$}\kern0.2pt}
}%
\DeclareRobustCommand\lowerstar{%
  \mathchoice%
    {\kern0pt\raise-0.65ex\hbox{$\displaystyle *$}\kern0.8pt}
    {\kern0pt\raise-0.68ex\hbox{$\textstyle *$}\kern0.8pt}
    {\kern0pt\raise-0.55ex\hbox{$\scriptstyle *$}\kern0.4pt}
    {\kern0pt\raise-0.5ex\hbox{$\scriptscriptstyle *$}\kern0.2pt}
}%

\newcommand{\lowershriek}{_!}

\newcommand{\scalars}{\kat{End}(\ZZ_{/\e})}
\newcommand{\innerprod}[2]{\langle#1,#2\rangle}

\newcommand{\dlpullback}{\arrow[phantom]{dl}[very near start,description]{\llcorner}}

\newcommand{\ddlpullback}{\arrow[phantom]{ddl}[very near start,description]{\llcorner}}

\newcommand{\Id}{\operatorname{Id}}
\newcommand{\Map}{\operatorname{Map}}
\newcommand{\Det}{\operatorname{Det}}
\newcommand{\Adj}{\operatorname{Adj}}

\newcommand{\Grpd}{\kat{Grpd}}

\newcommand{\grpd}{\FF}

\newcommand{\spand}{\kat{span}}

\providecommand{\norm}[1]{|\kern0.5pt{#1}\kern0.5pt|}
\providecommand{\Norm}[1]{\left| {#1}\right|}
\renewcommand{\epsilon}{\varepsilon}

\newcommand{\sign}{\operatorname{sign}}

\usepackage{scalerel}

\DeclareRobustCommand\Hak[1]{%
  \mathchoice%
    {\raise-0.3ex\hbox{\scaleto{\wedge}{9.2pt}}\kern-1.3pt^{#1}\kern-0.7pt}
    {\scaleto{\wedge}{7.4pt}\kern-0.9pt^{#1}\kern-0.4pt}
    {\scaleto{\wedge}{4.6pt}\kern-0.9pt^{#1}\kern-1.3pt}
    {\wedge\kern-0.7pt^{#1}\kern-0.7pt}
}%

\newcommand{\op}{^{\text{{\rm{op}}}}}

\newcommand{\Inj}{\operatorname{Inj}}

\usepackage{amsthm}

\newtheorem{lemma}{Lemma}[subsection]
\newtheorem{prop}[lemma]{Proposition}

\newtheorem{theorem}[lemma]{Theorem}
\newtheorem{cor}[lemma]{Corollary}

\newtheorem{taller}[lemma]{$\!\!$}

\newenvironment{blanko}[1]%
{\begin{taller}{\normalfont\bfseries  #1}\normalfont}%
{\end{taller}}

\begin{document}

\title{Signs in objective linear algebra, \\ exemplified with
exterior powers and determinants}

\author{{\sc Joachim Kock\thanks{Primary affiliation: Universitat Aut\`onoma de 
Barcelona and Centre de Recerca Matem\`atica} } and {\sc Jesper Michael M\o ller}}

\date{\small K\o benhavns Universitet}

\maketitle



\begin{abstract}
  We develop objective linear algebra in a new setting with a
cardinality functor that can take negative values. The signs
arise as little homotopies, as ratios between orientations. To
illustrate the workings of the theory we give an objective
treatment of exterior powers and determinants.
\end{abstract}

\blfootnote{
		\raisebox{-0.1pt}{\includegraphics[height=1.65ex]{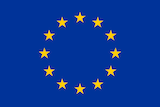}}
		This work was funded 2021--2023 by the European Union's Horizon 2020 research
        and innovation programme under Marie Sk\l odowska-Curie grant agreement 
        No.~101028099.\\
		\raisebox{-0.1pt}{\includegraphics[height=1.66ex]{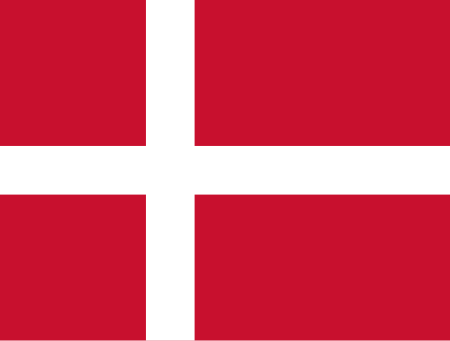}}
		This work was funded 2023--2025 by the Independent Research Fund 
		Denmark under grant  
        No.~10.46540/3103-00099B.\\
        Secondary funding sources: grant PID2024-158573NB-I00 (AEI/FEDER, UE) of Spain and grant
  2021-SGR-1015 of Catalonia, as well as the Severo Ochoa and Mar\'ia de
  Maeztu Program for Centers and Units of Excellence in R\&D grant number
  CEX2020-001084-M and the Danish National Research Foundation
  through the Copenhagen Centre for Geometry and Topology (DNRF151).
		}

        \begin{center}

        \begin{minipage}{120mm}
                       \small
    \tableofcontents
        \end{minipage}
        
        \vfill
        
\end{center}

\section{Introduction}

\subsection{Background}

{\em Motivation: From numbers to structure.} The elementary arithmetic
operations of the natural numbers -- addition, multiplication and exponentiation --
are shadows of categorical operations with finite sets -- disjoint union, cartesian
product and mapping sets. Working at the objective level (working with objects
themselves instead of their number) to the extent possible has several benefits. First
of all, functorialities and universal properties can be exploited to reveal more
details. Contrary to numbers, sets have elements which can be accessed individually,
sets can support further structure, and they have symmetries, leading to homotopy
theory. Furthermore, at the objective level, one can actually work with infinite sets
instead of just finite ones, thus reaching realms that are beyond the scope of
numbers. The passage from natural numbers to finite sets was exploited recently by the
first-named author to solve a longstanding open problem in theoretical computer
science, namely the unification of algebraic and geometric operational semantics of Petri
nets~\cite{Kock:2005.05108}.

A major limitation of the objective viewpoint is the lack of negatives, essential in
algebra not just for additive inversion, but also for most other kinds of inversion,
such as matrix inversion or power series inversion. It is thus a fundamental challenge
to provide a combinatorial model for negative numbers, giving a meaning to the
illusive notion of negative sets -- sets with a negative numbers of elements. A couple
of approaches have been proposed over the past 50 years, as explained in more detail
in the `Related work' section below, but the lacking categorical features of these
attempts prevented them from developing very far. One very desirable feature
of any category containing negative sets is that it should be sufficiently topos-like,
so that objective linear algebra can be developed over it.

{\em Objective linear
algebra}~\cite{Baez-Hoffnung-Walker:0908.4305},~\cite{Galvez-Kock-Tonks:1602.05082} is
basically linear algebra where the scalars are finite sets instead of numbers. Vector
spaces are replaced by slice categories; linear maps given by matrices are replaced by
colimit-preserving functors, in turn given by spans, so that matrix multiplication
becomes pullback-composition of spans. Within this setting, one can develop an
objective theory of certain basic parts of algebraic combinatorics, such as incidence
algebras and
bialgebras~\cite{Galvez-Kock-Tonks:1512.07573,Galvez-Kock-Tonks:1512.07577,Galvez-Kock-Tonks:1512.07580,Galvez-Kock-Tonks:1612.09225}
and generating functions~\cite{Joyal:1981,Joyal:foncteurs-analytiques},
\cite{Gambino-Kock:0906.4931}, \cite{Gepner-Haugseng-Kock:1712.06469}. Importantly,
the theory works more generally with scalars in groupoids or $\infty$-groupoids. This
extension allows to take into account symmetries and model nonnegative rational
coefficients. The important requirement for this extension is that the base category
be topos-like: at least it should have finite limits and colimits and be locally
cartesian closed and extensive (which implies distributive). There is a (homotopy)
cardinality functor which takes all objective notions to their numerical counterpart.
Altogether it is a systematic theory for lifting algebra with natural numbers
(resp.~nonnegative rationals) to the objective level, and it is a blueprint for what
we would like to have also for the integers (resp.~the rational numbers).

In this work we introduce a notion of `negative sets' meeting the topos-like
desideratum, and develop some objective linear algebra with signs. We get as far as
covering exterior powers and determinants (as well as matrix inversion, to some extent).
To this end we exploit some elementary homotopy theory: we find it necessary to
work with finite groupoids instead of just finite sets, and in the end the signs
arise as little homotopies.

Before outlining the approach, a comment is due as to why homotopy theory must 
necessarily be involved and the difficulty in achieving a 
topos-like category.

{\em Group completion and the sphere spectrum.}
 Just as the integers are constructed from the natural numbers by 
formally inverting the plus operation, one can pose the question of formally inverting
the disjoint-union operation on finite sets. The answer -- group completion -- is a 
central concept in homotopy theory.
It is a classical theorem
(Quillen, Barratt--Priddy~\cite{Barratt-Priddy}) that the group completion
of the groupoid of finite sets ${\B}$ is the {\em sphere spectrum} ${\S}$
(which is an $\infty$-groupoid rather than an ordinary groupoid).

Importantly, $\S$ has a multiplication (i.e.~is a ring spectrum) which
extends the product operation on $\B$ (just as the multiplication of $\Z$
extends that of $\N$). In fact, $\S$ is the initial
ring spectrum, just as $\Z$ is the initial ring in ordinary algebra.
This led to the slogan {\em the sphere spectrum is
the true integers}. Whatever the definition of signed finite sets, it must
relate meaningfully to the sphere spectrum.

In spite of its important universal property, the sphere spectrum is not quite
the solution to the challenge of finding a combinatorial notion of negative sets, because fundamental combinatorial features are lost
in the passage from $\B$ to $\S$.
For example, the map $\B \to \S$ collapses permutations to their sign: a
finite set $\underline{n}$ is sent to the object $n\in \pi_0(\S)= {\Z}$,
and a permutation $\tau:\underline n \isopil \underline n$ is sent to its
sign, $\epsilon(\tau) \in \pi_1({\S}) = {\Z}_2$. Furthermore, 
like the groupoid $\B$, the sphere spectrum does
not have anything like
surjections or injections (crucial in combinatorics for example for partitions
or binomial coefficients, respectively), and no meaningful slices for use in 
linear algebra. These features 
pertain to the {\em
category} $\F$ of finite sets and {\em all} maps.
There is a notion of group completion also for
monoidal categories that are not just groupoids, but one can show that inverting the
monoidal operation automatically also inverts arrows~\cite{Schanuel:negative}: a group
object in $\kat{Cat}$ is automatically a groupoid.
These facts indicate that group completion is too drastic an operation for the 
purpose of combinatorics.

Altogether there is a tension between combinatorics-like (topos-like) features of
$\F$, and algebra-like features of $\S$, and this tension is the main difficulty
in the whole enterprise. It is conceivable that the tension cannot be resolved 
uniformly.
and that different models for negative sets may be required for different applications
(As argued by Kapranov~\cite{Kapranov:1512.07042}, signs arise for 
various reasons, some of which
relate to the signs in $\Z=\pi_0 \S$, $\Z_2 = \pi_1(\S)$ and $\Z_2 = \pi_2(\S)$.)
But even without a uniform solution, it is still an important
endeavour to see how far one can stretch combinatorics into algebra.

\subsection{Contributions}

{\em Novelties in this work.}
The goal of this paper is to develop objective linear algebra in a setting where
there is a cardinality functor taking negative values.
We manage to resolve the above-mentioned tension by (1) working in a topos-like category
where slices, spans, pullbacks are
meaningful so that we can develop objective linear algebra comfortably, but (2) not
assigning negative cardinalities to its objects, but only to its scalars --
this notion makes sense relative to a monoidal structure which is not the cartesian 
product.
In ordinary linear algebra, a vector can be written as a linear combination of basis
vectors, such as $\sum_{b\in B} \lambda_b \basis_b$, where $B$ is the set indexing the
basis vectors $\basis_b$, and the $\lambda_b$ are scalars. So a vector is essentially
a $B$-indexed family of scalars. At the objective level, a `vector' is a map $E{\to}
B$ interpreted as a (finite) family of finite sets $(E_b\mid b\in B)$ (and the
linear-combination interpretation is the fact that the map $E{\to}B$ is isomorphic to
the sum $\sum_{b\in B} E_b \cdot (1{\stackrel{\name{b}}\to}B)$). We see that $B$ does
not need signs: we are not (currently) interested in considering a basis with a
negative number of elements. Where we need the signs is in the scalars, which arise as
fibres of a map.

In linear algebra over $\F$ there is no difference between scalar and object.
In the setting we develop, the main point is that such a difference arises: a scalar
is a linear endomorphism of the `ground field', and it will be slightly more than
just an object -- it will be an object together with some extra homotopy data which
we interpret as a sign.

Although the signs are present implicitly in the whole set-up,
they only really manifest themselves in the act of taking cardinality.
Herein lies one of the main discoveries: it is a cardinality functor, which at first sight 
looks a bit mysterious, but which turns out to enjoy a universal property.
The upshot is that although there may be some equal quantities of `positive' and
`negative', which one might think should cancel out, we do {\em not} divide out
by this cancellation, pretending it is the empty set. Rather this `irrelevant
junk' is kept through all the calculations, and only disappears when taking
cardinality. Keeping the junk and letting it take care of itself is what makes it
possible to achieve the topos-like category, whereas dividing out by it too
early would destroy important categorical properties.

\bigskip

Here is our implementation of these ideas. We need to work with finite groupoids 
instead of finite sets. Let $\FF$ denote the category of finite groupoids.\footnote{It is 
of course properly a $2$-category, and we will need the $2$-cells, but since it has only 
invertible $2$-cells, we will continue to write just `category'.}
(Finite $\infty$-groupoids (i.e.~$\pi$-finite spaces) are
also fine, but to stay as close as possible to finite sets, we stick with groupoids.)
The base category is the weak slice $\ZZ:=\FF_{/\P}$, where $\P = B O(1)$ is 
the classifying space of the cyclic group with two elements, $O(1) = C_2 = \{\pm1\}$.
It is the homotopy type of infinite real projective space.
Objects in $\ZZ$ are thus $f:S\to\P$, finite groupoids $S$ over $\P$, meaning 
that every arrow in
$S$ has a {\em parity}, namely even or odd. We insist on the terminology `even' and `odd' 
to stress that these quantities are not yet the signs we are after.
Parities are the structure with respect to which we will define 
orientations; the signs will finally arise as ratios between 
orientations.

Objective linear algebra works over $\ZZ$ quite similarly to how it works over $
\FF$. This leads to the category $\linpm$ whose objects are slices $\ZZ_{/f}$
(for all finite groupoids-with-parity-structure $f:S\to\P$) and whose morphisms are given by spans
over $\P$. 

The usual base category $\FF$ embeds into $\ZZ$ by giving all arrows in a 
groupoid even parity.
The category $\ZZ$ has a monoidal structure $*$ which extends the cartesian product
in $\FF$, but which crucially is not itself the cartesian-product structure.
The neutral object for $*$ is the trivial parity structure $\e:1 \to \P$ (which 
is not the terminal object in $\ZZ$).
The monoidal structure $*$ induces the relevant monoidal structure 
$\circledast$ on $\linpm$,
so as to get a (strict) monoidal functor
$$
(\lin, \tensor, \FF) \to (\linpm, \circledast, \ZZ_{/\e} ) ,
$$
as required to extend ordinary objective linear algebra.

Here comes the first important point. We define a scalar to be a linear endofunctor
of $\ZZ_{/\e}$, that is, a span from $1$ to $1$ over $\P$. When we say over $\P$,
it is of course in the homotopy sense, which means that there is a $2$-cell
\[
\begin{tikzcd}[row sep={2em,between origins}, column sep={4.8em,between origins}]
  & S \ar[ld] \ar[rd] &  \\
  1 \ar[rd, "\e"'] & \rho & 1   . \ar[ld, "\e"] \\
  & \P &
\end{tikzcd}
\]
This $2$-cell amounts to a disjoint-sum splitting of the scalar $S$ into
two parts
$$
S = S_\pp \sqcup S_\mm ,
$$
so that scalars finally have a positive and a negative part.
We declare the cardinality of such a scalar $S$ to be
$$
\norm{S} := \norm{S_\pp} - \norm{S_\mm} .
$$
A more general linear functor (an $(I,J)$-indexed matrix over $\P$) 
will now have matrix entries which are scalars, and hence will have a cardinality 
which is a matrix over $\Q$, including signs.

These assignments should assemble into
a functor
$\norm{\ } : \linpm \to \vect_\Q $. In the groupoid case, this insight is due to
Baez, Hoffnung, and Walker~\cite{Baez-Hoffnung-Walker:0908.4305}.
In ordinary objective linear algebra~\cite{Galvez-Kock-Tonks:1602.05082}, the relevant
assignment is
\begin{eqnarray*}
  \lin & \longrightarrow & \vect_\Q  \\
  \FF_{/I} & \longmapsto & \Q_{\pi_0 I}
\end{eqnarray*}
sending the slice category over $I$ to the vector space spanned by the
connected components of $I$. In the signed case, the same assignment
could be made, but it would not be the correct one. To explain the
correct assignment we exploit the notion of orientation. An {\em orientation}
of a groupoid with parity structure $S \to \P$ is a factorisation through
$1 \to \P$ (the universal cover of $\P)$. A connected component is
orientable if and only if its points have no odd automorphisms. The
orientable locus of a groupoid $I$ is the full subgroupoid $I^\circ
\subset I$ consisting of the orientable points. A groupoid playing the
role of a scalar is manifestly orientable (in fact the structure of
scalar amounts to the ratio of two orientations). We show that if a point
$i\in I$ is not orientable then the scalar corresponding to any $i$-index
of any matrix $I {\leftarrow} A {\to} J$ will have cardinality zero, due
to a cancellation constructed from any odd automorphims of $i$.
Accordingly, such an unorientable component should not count as a basis element. 
In the end, the correct
cardinality assignment extending that of scalars is
\begin{eqnarray*}
  \linpm & \longrightarrow & \vect_\Q  \\
  \ZZ_{/I} & \longmapsto & \Q_{\pi_0 (I^\circ)}  .
\end{eqnarray*}

A certain amount of work is required to show that this is functorial.
The subtlety is to show that certain equivalences of groupoids (known facts from
objective linear algebra without signs) are in fact also equivalences of {\em scalars}, so
that signs can be transferred correctly along them. A key result in this direction is 
Corollary~\ref{cor:iMjNk} stating an equivalence of scalars
  $$
  {}_i(X\times_J Y)_k \simeq \int_{j\in J} ( {}_i X_j \times {}_j Y_k) .
  $$
This equivalence is the signed-objective analogue of the fundamental fact that
composition of linear maps is given by matrix multiplication.

The restriction to the orientable loci may look surprising at first.
But in fact this cardinality functor can be characterised by a
universal property: it is the unique $1$-semiadditive extension of the functor $\P \to
\vect_\Q$ that sends the point to $\Q$ and sends $\pm1$ to $\pm\Id$, invoking a
theorem of Harpaz~\cite{Harpaz:1703.09764}. (This characterisation of our cardinality
functor was suggested to us by Maxime Ramzi.)

A more pragmatic justification for this cardinality functor (and the reason we 
discovered it) is that it is what makes the notion of exterior power work.
We define the $k$-th exterior power of a finite set or groupoid $S$ to be the 
groupoid-with-parity-structure
$$
\Hak k S := (\textstyle\frac{S^k}{k!}{\,\to\,}\P) .
$$
Here the groupoid $\frac{S^k}{k!}$ is the weak quotient
$\Map(k,S)/\Sigma_k$ and importantly its parity structure is induced by
the sign representation $\Sigma_k\to O(1)$. The orientable points of
$\frac{S^k}{k!}$ are precisely the injective maps $k \to S$. It follows
that if $S$ has $n$ orientable components, then $\Hak k S$ has $n \choose k$
orientable components, as it
should. The expected properties of $\Hak k(\ZZ_{/S}) := \ZZ_{/\Hak k S}$
now follow from standard properties in $\linpm$. In particular we show 
that $\Hak k(\ZZ_{/S})$ has a universal property as recipient of 
alternating linear functors.

With the notion of exterior power sorted out, we can define determinants.
Given a matrix in the form of an endospan $I \leftarrow A \to I$ (which could be
in $\linpm$ or just in $\lin$) with $\#(\pi_0 I^\circ)= n$, the determinant of $A$, denoted 
$\Det(A)$, is
defined as the span $\Hak n I \leftarrow \Hak n A \to \Hak n I$.
This span may be huge and contain a lot of immaterial junk, but its orientable part 
is a $1$-by-$1$ matrix, i.e.~a scalar in the technical sense, so its cardinality is a
rational number. This works out as required:
$$
\norm{\Det (A)} = \det \norm{A} .
$$
We derive this result from a fancier result at the objective level: we 
establish a natural equivalence of groupoids-with-parity-structure (in 
fact an equivalence of scalars):
  $$
  {}_{\underline x} (\Hak n A)_{\underline x} \simeq 
  \sum_{\sigma\in\Sigma_n} \sign(\sigma) \prod_{i\in n} {}_{x_i} 
  A_{x_{\sigma i}}  ,
  $$  
  which is an objective version of the 
Leibniz expansion. 
The previous formula results from taking cardinality of this equivalence.

\bigskip

We can also define the cofactor (or adjugate) of a span, using 
next-to-top exterior powers $\Hak{n-1}A$, and establish
a natural equivalence of groupoids
$$
A \circ \Adj(A) \simeq \Det(A) \cdot \Id .
$$
The left-hand side is pullback-composition of spans (i.e.~matrix 
multiplication); the right-hand side is a suitable
notion of scalar times identity matrix. This is as close as we get to objective 
matrix inversion at this point. Since the  
whole theory of cofactor spans is a longer story, we have decided not to 
include it in the present paper.

\subsection{Related work}

The basic idea of signed finite sets is to consider pairs of finite
sets $(A,B)$, thought of as `$A$ minus $B$', with rules for how the positive and
the negative interact. It is a big challenge to make this work. For example, if
the distributive law is to be maintained, such a category will collapse
completely~\cite{Schanuel:negative}. It is possible to develop a certain
calculus for such signed sets, the most famous one being the involution
principle of Garsia and Milne~\cite{Garsia-Milne:1981}, who used it to prove
some deep combinatorial identities in number theory.
Zeilberger~\cite{Zeilberger:1985} and others used this calculus to give
bijective proofs of some fundamental matrix identities. 
Conway and Doyle~\cite{Doyle:1907.09015}, \cite{Doyle-Conway:math/0605779},
and others gave an interpretation of the involution principle in terms of
versions of the $1$-dimensional cobordism category. An important 
interpretation is the (older) result of Joyal, Street, and Verity~\cite{Joyal-Street-Verity}
that this category is the free compact closed category on the traced 
monoidal groupoid $\B$. The trace on $\B$ is given by {\em residual bijection},
the remarkable fact that any bijection of finite sets $A+X\simeq A+Y$ 
induces a canonical bijection $X\simeq Y$ 
(see~\cite{Joyal:foncteurs-analytiques}).

Another important contribution is the theory of virtual species of
Joyal~\cite{Joyal:foncteurs-analytiques}, which achieved an objective
account of power-series inversion. It features what is now called Joyal's
sign rule, where $\exp(-X)$ is interpreted as an alternating sum of
surjections. Loeb~\cite{Loeb:1991} studied negative sets in terms of
multisets with $\Z$-valued multiplicities, called hybrid sets, and used
it to give a combinatorial interpretation of binomial coefficients with
negative entries. The signs are introduced in a way that make the
formulae come out correctly, but it is probably fair to say that the
signs themselves are not really explained. The approach is a mix of
combinatorics and algebra, where the signs are simply imported from
algebra. While these theories have been important and still play a role
in combinatorics, it is difficult to develop them very far, because of
their lacking categorical properties. Neither virtual species nor hybrid
sets form a usable category. In particular, these approaches are not
suitable for developing objective algebraic combinatorics. The important
viewpoints and insights they represent should eventually be subsumed in a
more unified theory of negative sets.

Negative numbers can also be made arise from Euler characteristics
instead of insisting on cardinality, and one may try to use objects that
admit an Euler characteristics as a stand-in for negative sets. From the
viewpoint of combinatorics, a natural such context would be that of
finite CW-complexes or finite categories~\cite{Leinster:0610260}. It is
worth to call attention also to Schanuel's negative sets~\cite{Schanuel:negative}
in terms of Euler characteristics of polyhedral sets in euclidean spaces.
Here the open interval $(0,1)$ is given (Schanuel)--Euler characteristic
$-1$, as it is obtained from the closed interval $[0,1]$ (Euler
characteristic $1$) by removing two points. This notion
satisfies all the expected axioms of Euler characteristics, except
homotopy invariance. It is certainly an intriguing notion, which deserves
further study, but it does not quite
fit the needs of objective linear algebra.

\bigskip

\noindent {\em Kapranov's work.} The relationship between signs and the
sphere spectrum has been analysed deeply by Kapranov. In his remarkable
paper~\cite{Kapranov:1512.07042}, he first explains (following an insight
of Joyal) how super\-geometry in mathematics and the Koszul sign rule
originate with how $\Z_2$ sits inside the sphere spectrum as
$\pi_{1}(\S)$, and then carries on to explain how in physics the notion
of supergeometry (rather about spinors and square roots) relates instead
to $\Z_2 = \pi_2(\S)$. The two instances of the group $\Z_2$ containing
the signs thus relate differently to $\S$. Next he analyses the interplay
between these two homotopy levels, and finally argues that a deeper
understanding of signs should involve the full sphere spectrum.
Kapranov's work is a main source of inspiration, but his paper is not
concerned with the question we address in this work. While Kapranov's
work analyses the behaviour of signs, taking for granted the notion of
vector space (and negative scalars), the present work operates at a more
fundamental level, aiming actually
to construct the signs combinatorially.%
\footnote{Similarly, Ganter and Kapranov~\cite{Ganter-Kapranov:1110.4753}
studied `exterior powers of categories', but their categories are 
assumed from the outset to be
linear categories (meaning enriched over vector spaces), which means that
signs and negative quantities are already available to begin with.}

\subsection{Outlook}

In this paper we wish to concentrate on the notion of signs, and we
develop the theory with that specific focus. However, most of our
constructions in Sections 2--3 admit a generalisation to the case where
$\P$ is replaced by $BA$, the classifying space of an abelian group $A$.
We are not completely sure about the significance of the full generality,
but the special case of cyclic groups and the limit group $\Q/\Z$ seems
to be very promising for the prospective of getting to some objective
version of the complex numbers. Instead of parities we then talk about
{\em phases}. In a sequel paper we will pursue that direction.

Another direction for further work concerns replacing $\P$ with the full
homotopy type of the zero-component of the sphere spectrum, rather than
only its $1$-truncation. For the aspects of the theory developed so far,
it does not make much of a difference, as in any case the cardinality
functor will collapse higher homotopies, but in the light of Kapranov's work, 
it is certainly a line of inquiry worth further effort.

We wish to emphasise that
the perceived success with modelling exterior powers and determinants
should not be taken as universal support for the viability of the model
for signed sets introduced in this paper. After all, the model was essentially
designed to make this work. It is not clear how the notion will fare in
modelling other sign phenomena in algebraic combinatorics. In particular,
signs play a key role in theory of M\"obius
inversion, and in the objective account~\cite{Galvez-Kock-Tonks:1512.07573},
\cite{Galvez-Kock-Tonks:1512.07577}, \cite{Galvez-Kock-Tonks:1512.07580},
the signs are handled in a rather ad hoc way, by `moving negatives to
the other side of the equation'. At the moment we have not been able to take any 
advantage of the new model in that context.

A pressing project is the general study of cancellations. The realisation
that the algebraic notion of zero is not the same thing as nothing, but
that there can be huge chunks of matter with net cardinality zero, and
that there are structures exhibiting the cancellations, is an idea that
resonates well with modern homotopy theory, which is increasingly seen as the
science of different ways things can be equal. Learning how to deal with
such cancellations in a controlled way, especially in connection with the
infinite, could potentially have big implications for topics such as
divergent series and perturbative renormalisation.

\subsection{Organisation of the paper}

The paper has three sections after the introduction. Section 2 develops linear algebra
over $\ZZ$ in a purely objective fashion, without mention of cardinality. It is a
feature that not many signs are seen at this level -- they are sort of implicit and
take care of themselves -- but quite a lot of work is required to show that certain
equivalences of groupoids are also equivalences of scalars. In the very short Section
3 we introduce the cardinality functor. Since we have done most of the work at the
objective level, there is not so much to do here, but it is still an important aspect,
especially in view of the surprising nature of the cardinality functor. Finally in
Section 4 we work out what exterior powers and determinants are in our framework. Most
of the work here consists of verifications, to justify that the rather surprising
definitions are reasonable.  We have sign-posted four turns along the way.

\bigskip

  \noindent {\bf Acknowledgments.} We wish to thank Gon\c calo Marques,
  Andr\'e Joyal, Maxime Ramzi, Rune Haugseng, Louis Martini, Imma
  G\'alvez, Andy Tonks, Jan Steinebrunner, Anders Kock, Thomas Jan
  Mikhail, John Baez, and Sebastian Wolf for their feedback
  and help.


\section{Linear algebra}

\subsection{Objective linear algebra without signs}

\label{sub:lin}

The basic theory of objective linear algebra is now standard. A systematic treatment
is given in \cite{Galvez-Kock-Tonks:1602.05082}. We quickly run through the main
points to fix notation and set the stage.

\begin{blanko}{Linear algebra with coefficients in finite sets.}
  Let $\F$ denote the category of finite sets (and all maps).
  Objective linear algebra over $\F$ works by regarding slice categories 
  $\F_{/S}$ as a kind of
  vector spaces.
  The reason $\F_{/S}$ behaves like a vector space is that it is 
  the finite-sum completion of the discrete category $S$, just like the 
  vector space
  $\Q_S$ spanned by $S$ is the `linear-combination completion' of $S$.
  An object $f: X{\to}S$ in $\F_{/S}$ thus plays the role of a vector.
  It can be written as a sum of its fibres. Clearly $X \simeq \sum_{s\in 
  S} X_s$, but more precisely we have
  $$
  f:X{\to}S  \qquad \simeq \qquad  \sum_{s\in S} X_s \cdot \name{s} \,,
  $$
  where $\name{s}: 1 \to S$ is the map that picks out element $s$. These 
  maps $\name{s}$ thus serve as a basis.
  
  The connection to ordinary linear
  algebra (over $\Q$) is via cardinality: the cardinality of the slice
  category  $\F_{/S}$ is the vector space $\Q_S$ spanned by $S$, and the
  cardinality of an object $X \to S$ in the slice is the vector in $\Q_S$
  given as the linear combination
  \begin{equation}\label{eq:lincomb}
  \sum_{s\in S} \norm{X_s} \cdot \basis_s.
  \end{equation}
  (Here $\basis_s = \norm{\name{s}}$ is the basis vector corresponding to the 
  element $s\in S$.)
  
  A functor between slices is called {\em linear} if it preserves finite 
  sums. From the universal property of $\F_{/S}$ as the sum 
  completion, it is not difficult to show that every linear functor 
  $\F_{/S} \to \F_{/T}$ is given (essentially uniquely) by a span 
  $$S\stackrel p\leftarrow M \stackrel q \to T.$$
  The corresponding functor is  
  $$
  \F_{/S} \stackrel{q\lowershriek \circ p \upperstar}\longrightarrow 
  \F_{/T},
  $$
  pullback along $p$ followed by composition along $q$.
  The result is the objective 
  version of the basic fact in linear algebra that for fixed bases, 
  every linear map is given uniquely by  a matrix.

  Composition of spans is given by pullback (and Beck--Chevalley shows
  that this corresponds to composition of functors).
  The span can be interpreted as a matrix of finite sets, namely 
  the two-sided fibres ${}_sM_t$, and the classical statement that composition of
  linear maps is matrix multiplication has the following objective analogue,
  again a question of splitting
  into fibres:
  $$
  {}_i(X\times_J Y)_k \simeq \sum_{j\in J} ( {}_i X_j \times {}_j Y_k)
  $$

  Taking cardinality produces the classical notions, as follows readily from the fibre
  formulae: the cardinality of a linear functor is the linear map given by the matrix
  of cardinalities of the two-sided fibres of the span.
  
  The tensor product of two slices is given by
  $$
  \F_{/S} \otimes \F_{/T} = \F_{/(S\times T)}
  $$
  in tight analogy with $\Q_S \otimes \Q_T = \Q_{S\times T}$.
  Altogether there is a symmetric monoidal category $(\lin, \otimes, 
  \F)$, whose objects are slices of $\F$ and whose morphisms are
  linear functors (given by spans), and the cardinality assignments 
  assemble into a functor
  $$
  \norm{ \ \cdot \ } : \lin \to \vect
  $$
  which is symmetric monoidal.
\end{blanko}

\begin{blanko}{Linear algebra over finite groupoids.}
  Over the $2$-category $\FF$ of finite groupoids, the situation is very 
  similar, provided all notions are taken in their homotopy version.
  We thus consider homotopy slices, and exploit that the homotopy slice  $\FF_{/S}$
  is the finite-homotopy-sum completion of the groupoid $S$. {\em Homotopy sum}
  means (homotopy) colimit over a groupoid, just as ordinary sum means 
  colimit over a set. A model for the homotopy colimit of a functor $F: S \to \FF$
  is the Grothendieck construction $\int_S F$, which is 
  equivalent to $\sum_{s\in \pi_0 S} F(s)/s!$, 
  where the quotient is the homotopy quotient for the natural action of
  $$
  s! := \Aut(s)
  $$
  on $F(s)$. 
  \footnote{The factorial notation $s! := \Aut(s)$ is both natural and practical. The symmetry
  factors $/s!$ appear in the same places as they do in combinatorics, and in each
  case they can be interpreted as a homotopy quotient of a group action, and usually 
  appearing together with sums $\sum_{s\in \pi_0 S}$ so as to form together colimit 
  data rather than coefficient data.
  This is analogous to how coefficients of exponential generating functions are 
  interpreted combinatorially: in $\sum_i a_i \frac{z^i}{i!}$, the coefficient is 
  $a_i$, not $a_i/i!$; the $\sum_i$ and the $/i!$ arise together as 
  homotopy sum over the groupoid $\mathbb{B}$ of finite sets and bijections. This is the
  lesson from species theory~\cite{Joyal:1981}, 
  \cite{Joyal:foncteurs-analytiques} (although the interpretation as a homotopy 
  sum is only possible after passing to groupoid-valued 
  species~\cite{Kock:MFPS28}, \cite{Gepner-Haugseng-Kock:1712.06469}).}

  Starting with a general map $X{\to} S$, we get a
  functor $F: S\to \FF$ sending $s$ to the homotopy fibre $X_s$, and we 
  obtain the canonical decomposition of $X{\to}S$ as a homotopy sum of
  its fibres
  \begin{equation}\label{int-of-fib}
  X{\to}S  \qquad \simeq \qquad  \int_{s\in S} X_s \cdot \name{s} 
  \ \simeq \ \sum_{s\in \pi_0 S} \frac{X_s \cdot \name{s}}{s!}
  \ \simeq \ \sum_{s\in \pi_0 S} X_s/s! \cdot \name{s}
  \end{equation}
  with reference to the natural action of the group $s!$
  on the map $X_s \cdot \name{s}$
  or on the homotopy fibre $X_s$.  Again we see that the map 
  $\name{s}:1 \to S$ serve as a basis.
  
  The connection to ordinary linear algebra over $\Q$ is now via
  homotopy cardinality. Recall (for example from \cite{Baez-Dolan:finset-feynman}) 
  that the (homotopy) cardinality of a groupoid is defined 
  as
  $$
  \norm{S} := \sum_{s\in \pi_0 S} \frac{1}{\norm{s!}} .
  $$
  The cardinality of $f{:}X{\to}S$ is now
  $$
  \norm{f} = \sum_{s\in\pi_0 S} 
  \frac{\norm{X_s}}{\norm{s!}} \basis_s  \quad \in \Q_{\pi_0(S)} .
  $$
  
  A functor between slices $\FF_{/S} \to \FF_{/T}$ is {\em linear} if
  it preserves finite homotopy sums, and again one can show that these
  are precisely given by spans $S \stackrel p \leftarrow M \stackrel q 
  \to T$ by the formula $q\lowershriek \circ p \upperstar$.
  The formula for matrix multiplication now involves a homotopy sum instead of a 
  discrete sum:
    $$
  {}_i(X\times_J Y)_k \simeq \int_{j\in J} ( {}_i X_j \times {}_j Y_k) .
  $$

  Again the cardinality assignment extends to linear functors, but there
  is a very important twist: the entries of the matrix representing the
  cardinality of the linear functor are not just $\norm{{}_s M _t}$
  but rather 
  \marginpar{\dbend} 
  $$
  \norm{{}_sM_t}/\norm{t!}.
  $$   
  The necessity of these symmetry factors $/t!$
  is a fact of life in homotopy linear algebra
  (see~\cite{Galvez-Kock-Tonks:1602.05082}). It comes about because the
  entries of the matrix can be characterised as the coefficients of the
  linear combination that expresses the image of a basis vector of the
  domain in terms of the basis on the codomain. But in the homotopy case,
  this has to be a homotopy sum, which involves these denominators, and
  they appear in the codomain and not in the domain for this reason.
  
  The tensor product $\FF_{/S}\otimes \FF_{/T} = \FF_{/S\times T}$ and 
  the symmetric monoidal cardinality functor 
  $$
  \norm{ \ \cdot \ }:
  (\lin, \otimes,\FF) \to (\vect, \otimes,\Q)
  $$
  work the same as in the discrete case. Note that we update the symbol
  $\lin$ to refer now to groupoid slices.
\end{blanko}

\begin{blanko}{Example.}
  To illustrates clearly why the symmetry factors $/t!$ are needed in the
  matrix entries, consider the identity span $T\stackrel=\leftarrow T
  \stackrel=\to T$. Then the diagonal entries of the matrix will come
  from the two-sided fibres ${}_t T_t$, but these fibres are $\Omega B
  \Aut(t)$ (the loop space of $T$ at $t$). The cardinality of this
  two-sided fibre is $\norm{t!}$, the number of automorphisms of
  $t$. This alone would not be the identity matrix. But the group
  $t!$ acts on this loop space, and the homotopy quotient is
  contractible, since the action is both free and transitive, hence
  giving the desired diagonal entry $1$. So with the symmetry factor $/t!$ included
  we do get the the identity matrix.
\end{blanko}

In the rest of the paper we aim at a similar story in a context 
where the cardinality can have signs. The way we set it up, the main steps from the 
set case and the groupoid case carry over smoothly for general reasons, but it will
require some work to actually see the signs, and see that they come out correctly 
under the cardinality functor we define.

\subsection{Parities and the `ground field'}

Denote by $O(1)$ the multiplicative group $\Z^\times = \{+1,-1\}$.
The corresponding one-object groupoid is denoted
$$
\P := BO(1) .
$$ 
(The letter P is chosen for `parities', but the symbol may also remind us that $\P$ has 
the homotopy type of infinite real projective space.)

A {\em parity structure} on a groupoid $S$ is by definition 
a morphism of 
groupoids
$$
f: S \to \P.
$$
It is thus an assignment of values $\pm 1$ to every arrow in $S$,
in a way compatible with composition of arrows.
These are {\em not} the signs that will correspond to positive/negative numbers,
and for this reason we prefer to call these values assigned to arrows {\em 
parities}, so that we speak of an {\em even} or an {\em odd} arrow.

Let $\grpd$ denote the $2$-category of finite groupoids, functors, and 
invertible natural transformations.
Finite groupoids with parity structure form a $2$-category
$$
\ZZ:=\grpd_{/\P},
$$
the weak slice. A morphism from $f: S \to \P$ to $g: T \to \P$ is thus a pair
$(u,\theta)$ where $u:S \to T$ is a morphism of groupoids (that is, a functor),
and $\theta$ is a homotopy (that is, a natural transformation)
$$
\begin{tikzcd}[column sep={2em,between origins}]
  S  \ar[rr, "u"] \ar[rd, "f"'] &
  \ar[d, phantom, pos=0.3, "\theta" description]
  & T \ar[ld, "g"] \\
  & \P . &
\end{tikzcd}
$$
A $2$-cell from $(u,\theta)$ to $(u',\theta')$ is a $2$-cell
$$
\begin{tikzcd}[column sep={3.8em}]
  S 
  \ar[r, bend left, "u"] 
  \ar[r, bend right, "u'"'] 
  \ar[r, phantom, "\simeq"] &
  T
\end{tikzcd}
$$
mediating between $\theta$ and $\theta'$.

\begin{blanko}{Example.}
  The terminal groupoid $1$ admits only one parity structure, which we denote by
  $$
  \e: 1 \to \P.
  $$
  Note however that this structure is isomorphic 
  with itself in two ways: there are two possible natural transformations $\epsilon$
  $$
  \begin{tikzcd}[column sep={2em,between origins}]
	1  \ar[rr, "="] \ar[rd, "\e"'] &
	\ar[d, phantom, pos=0.3, "\epsilon" description]
	& 1 \ar[ld, "\e"] \\
	& \P . &
  \end{tikzcd}
  $$
  depending on whether the unique component of $\epsilon$ is even or odd. In
  particular, the object $\e$ is not terminal in $\ZZ$. (The terminal object in
  $\ZZ$ is of course $\id: \P\to\P$, which will not play any important 
  role in the following.)
\end{blanko}

\begin{blanko}{Remark.}
  Note that a morphism of parity structures $(u,\theta)$ must send even
  automorphisms to even automorphisms. For arrows between distinct objects, it
  can happen that an even arrow is sent to an odd arrow, just as it is possible
  that odd arrows (also odd automorphisms) can be sent to even arrows.
\end{blanko}

\begin{blanko}{Base `topos'.}
  We are interested in developing objective linear algebra over $\ZZ:=\grpd_{/\P}$. 
  This is possible since $\ZZ$ is a pretopos in the sense of 
  $\infty$-categories, so as to have meaningful slices, fibres, and so on.
  In fact it is $2$-truncated, and we will work with it as 
  a $2$-category with only invertible $2$-cells.
  
  Doing linear algebra means that the basic objects are (weak) slices,
  such as
  $$
  (\grpd_{/\P})_{/f} =: \ZZ_{/f} =: \ZZ_{/S}
  $$
  for some $f: S \to \P$. These slices play the role of vector spaces.
  While setting up the theory it can be convenient to explicitly reference 
  $f$, but in the long run it seems more useful to suppress mention of $f$
  and write $\ZZ_{/S}$ instead of 
  $\ZZ_{/f}$, since all groupoids will be assumed equipped with 
  a parity structure. We will then simply write $\pari(a)$  generically
  for the parity of an arrow $a$ in $S$ or in any parity groupoid.
\end{blanko}

\subsection{Group actions and homotopy sums}

\label{sub:group-act}
Since homotopy sums (colimits over a groupoid $I$)
split into discrete sums of weak quotients of group actions:
$$
\int_{i\in I} X_i \simeq \sum_{i\in\pi_0 I} X_i / \Aut(i) ,
$$
and since this splitting is useful for the compatibility with cardinality,
group actions will play an important role throughout.
We briefly go through the definitions in $\ZZ$ and some key properties.

We follow the convention that right actions are {\em strict} functors $BG \to \CC$ (to some 
category), written $x\mapsto x.g$, whereas left actions are (strict) functors $BG\op\to\CC$,
written $x \mapsto g^{-1}.x$.

\begin{blanko}{Actions on parity structures.}
  A $G$-action on a parity structure $f:X \to \P$ is a (strict) functor $BG \to \ZZ$
  with value $f$. The data of an action is thus given by
  $$
  \begin{tikzcd}[column sep={2em,between origins}]
    X  \ar[rr, ".g"] \ar[rd] &
    \ar[d, phantom, pos=0.3, "\theta_g" description]
    & X \ar[ld] \\
    & \P  &
  \end{tikzcd}
  $$
  for each $g$, subject to the compatibility
  $$
  \begin{tikzcd}[column sep={3.5em,between origins}]
    X  \ar[r, ".g"] \ar[rd, bend right] & X \ar[d] \ar[r, ".h"] & X \ar[ld, bend left]
    \\
    \ar[ru, phantom, pos=0.7, "\theta_g" description]
    & \P  &
      \ar[lu, phantom, pos=0.7, "\theta_h" description]
  \end{tikzcd}
  =
  \begin{tikzcd}[column sep={3em,between origins}]
    X  \ar[rr, ".(gh)"] \ar[rd, "f"'] &
    \ar[d, phantom, pos=0.3, "\theta_{gh}" description]
    & X \ar[ld] \\
    & \P  &
  \end{tikzcd}
  $$
  In the remainder of the paper,
  this condition will be referred to as `compatibility with composition' without 
  spelling it out in detail again.
  Note that the $2$-cells $\theta_g$ constitute extra data that is part of the 
  structure,
  and that it is a natural transformation, with a component $(\theta_g)_x$ for each 
  point $x\in X$.
  In many situations this data will be simply $\theta_g = \pari(g)$, in cases where
  we already have a parity structure $BG\to\P$; note that $\theta_g$ is then independent 
  of $x$.
\end{blanko}

\begin{blanko}{Weak quotients.}
  If a finite group $G$ acts on a parity groupoid $X{\to}\P$, the colimit in $\ZZ$
  can be calculated by first taking the weak quotient groupoid $X/G$, and then 
  describe it parity structure.
  So its objects are the objects of $X$, and the arrows are formal composites $
  x\to y \to y.g$.
  We refer to the arrows of the form $x\to y$ as {\em old} arrows, while the
  arrows of the form $y\to y.g$ are called {\em new} arrows. In the parity structure on
  $X/G$, all the old arrows retain their original parity, and the new arrows
  $y\to y.g$ get parity $(\theta_g)_y$.
\end{blanko}

\begin{blanko}{Example.}\label{ex:BG-as-colimit}
  Suppose we have a parity structure $BG\to \P$, and let $G$ act on the object 
  $\e:1{\to}\P$ by
  $$
  \begin{tikzcd}[column sep={3.2em,between origins}, row sep={3.6em,between origins}]
    1  \ar[rr, "\id"] \ar[rd] &
    \ar[d, phantom, pos=0.25, "{\scriptsize \pari(g)}"]
    & 1 \ar[ld] \\
    & \P  &
  \end{tikzcd}
  $$
  Then the weak quotient $\e/G$ is $BG$ with its original parity structure.
  This is a very pretenseless example, but it is the key ingredient in the proof of 
  the next two well-known results (see for example 
  \cite{Galvez-Kock-Tonks:1602.05082}).
\end{blanko}

\begin{lemma}\label{lem:colim-of-e}
  Every $X{\to}\P$ is (uniquely) the homotopy sum of copies of $\e:1{\to}\P$.
\end{lemma}
\begin{prop}
  $\ZZ$ is the finite-homotopy-sum completion of $\P$.
\end{prop}

\subsection{Orientations}

\label{sec:orientations}

The following notion of orientation is not logically necessary for the results in 
this paper, but conceptually it is very helpful.

An {\em orientation} on a parity structure $(S,f)$ is by definition a
morphism $(u,\omega)$
$$
\begin{tikzcd}[column sep={2em,between origins}]
  S  \ar[rr, "u"] \ar[rd, "f"'] &
  \ar[d, phantom, pos=0.3, "\omega" description]
  & 1 \ar[ld, "\e"] \\
  & \P  &
\end{tikzcd}
$$
to the trivial parity structure $\e:1 \to \P$, the universal (double)
cover of $\P$. The category of oriented parity structures is thus the (weak) slice
$\ZZ_{/\e}$.

\begin{lemma}\label{lemma:2^k orientations}
  A parity structure admits an orientation if and only if it has no
  odd automorphisms.  
  In that case there are $2^k$ possible orientations
  (for $k$ the number of connected components).
\end{lemma}
\begin{proof}
  If there is an odd automorphism $g: s \to s$, then it is not 
  possible to achieve naturality with respect to $g$: the naturality 
  square should be
  \[
  \begin{tikzcd}
  \cdot  \ar[r, "\omega_s"] \ar[d, "f(g)"'] & \cdot \ar[d, 
  "\e u(g)"]  \\
  \cdot \ar[r, "\omega_s"'] & \cdot
  \end{tikzcd}
  \]
  so there is no choice of $\omega_s$ to make it commute. 
  Conversely, if
  there are no odd automorphisms, then it follows that all arrows from $s$ to 
  $s'$ have the same parity. One can now easily exhibit an orientation by
  making a
  choice for $\omega_s$ for one point $s$ in each connected component. The 
  value in every other point in the same connected component is now forced.
  If there are $k$ 
  connected components, we thus get $2^k$ possibilities, as stated.
\end{proof}

\begin{blanko}{Example.}
  Note that the terminal object $\id : \P \to\P$ is {\em not} orientable.
\end{blanko}

\subsection{Linear maps}

\begin{blanko}{Linear functors.}
  The {\em linear functors} from $\ZZ_{/f}$ to $\ZZ_{/g}$ are given by spans
  over $\P$, such as
\begin{equation}\label{eq:BZ2span}
\begin{tikzcd}[row sep={2em,between origins}, column sep={4.8em,between origins}]
  & M \ar[dd, "m" description] \ar[ld, "p"'] \ar[rd, "q"] &  \\
  S \ar[rd, "f"'] &
  \ar[l, phantom, pos=0.3, "\stackrel{\epsilon'}\Leftarrow" description]
  \ar[r, phantom, pos=0.3, "\stackrel\epsilon\Rightarrow" description]
  & T \ar[ld, "g"] \\
  & \P &
\end{tikzcd}
\end{equation}
We write $\widehat M$ to refer to the whole span, including the $2$-cells.

The actual functor is
$$
\ZZ_{/f} \stackrel{q\lowershriek \circ p\upperstar}\longrightarrow \ZZ_{/g} .
$$
\end{blanko}

\begin{blanko}{Remark.}
  When the base category is $\FF$, one has for each $S$ the canonical equivalence
  $\FF_{/S}\simeq \FF^S$, which is the key ingredient in the proof of the (well-known)
  fact that the colimit-preserving functors between slices are precisely the functors
  given by
  spans (see for example \cite[2.10]{Galvez-Kock-Tonks:1602.05082}).
  
  Over $\ZZ$ the situation is more complicated, and the naive analogue of that
  statement is not true. Consider for example $s: BG\to\P$
  corresponding to a surjective group homomorphism and $t:BG\to\P$ corresponding to
  the trivial group homomorphism. Then as abstract categories both $\ZZ_{/s}$ and
  $\ZZ_{/t}$ are equivalent to $\FF_{/BG}$. Under these equivalences, the identity
  functor defines a functor $\ZZ_{/s} \to \ZZ_{/t}$, which takes an object $X{\to}BG$
  to the same object but with trivial parity structure. This functor preserves all
  colimits because it is an equivalence, but it cannot be realised by a $\P$-span,
  since the codomain is orientable and the domain is not.
  
  One way to fix this problem is to work internally to $\ZZ$
  (see 
  \cite{Martini-Wolf:2111.14495}). For this, one must
  first interpret all the slices $\ZZ_{/f}$ as categories internal to
  $\ZZ$; under this interpretation one has $\ZZ_{/f}\simeq \ZZ^f$
  \cite[Remark~3.7.2]{Martini:2103.17141},
  and then the usual proof
  that colimit-preserving functors are given by spans internalises,
  restoring the interpretation of ($\P$)-spans as precisely the (internal) 
  colimit-preserving functors.
\end{blanko}

\begin{blanko}{Remark.}
  Since all notions (slices, pullbacks, quotients, etc.) are the homotopy notions, one
  can work with the groupoids as one would with $\infty$-groupoids, and might not want
  to pay much attention to the $2$-cells -- in the $\infty$-world a commutative diagram
  is always only up to homotopy anyway. In the literature on $\infty$-categories, it
  is common not even to name these homotopies, but of course this happens {\em only}
  in situations where `nothing could go wrong'. In the present work we have to be
  extra careful with the $2$-cells, because it is where the signs are encoded (and
  experience shows that calculations with signs can actually go wrong). The following
  painstaking manipulations with $2$-cells should thus not be taken as pedantry.
\end{blanko}

\begin{blanko}{Morphisms of $\P$-spans.}\label{morphism-of-spans}
  Given two $\P$-spans
  \[
  \begin{tikzcd}[row sep={2em,between origins}, column sep={4.8em,between origins}]
   & M \ar[dd] \ar[ld] \ar[rd] &  \\
  S \ar[rd, "f"'] & 
  \ar[l, phantom, pos=0.3, "\stackrel{\epsilon'}\Leftarrow" description] 
  \ar[r, phantom, pos=0.3, "\stackrel\epsilon\Rightarrow" description]
  & T \ar[ld, "g"] \\
  & \P &
  \end{tikzcd}
  \qquad
  \text{and}
  \qquad
  \begin{tikzcd}[row sep={2em,between origins}, column sep={4.8em,between origins}]
   & N \ar[dd] \ar[ld] \ar[rd] &  \\
  S \ar[rd, "f"'] & 
  \ar[l, phantom, pos=0.3, "\stackrel{\delta'}\Leftarrow" description] 
  \ar[r, phantom, pos=0.3, "\stackrel\delta\Rightarrow" description] 
  & T  ,  \ar[ld, "g"] \\
  & \P &
  \end{tikzcd}
  \]
  a {\em morphism of $\P$-spans} from $\widehat M$ to $\widehat N$ is a quadruple
  $$
  (u,\theta, \sigma, \tau )
  $$
  with $u: M \to N$ a morphism of groupoids, and where $\theta$, $\sigma$, 
  and $\tau$ are $2$-cells
  \[
  \begin{tikzcd}[column sep={2em,between origins}]
	M  \ar[rr, "u"] \ar[rd] & \ar[d, phantom, pos=0.3, "\stackrel\Rightarrow\theta" 
	description] & N \ar[ld] \\
	& \P  &
  \end{tikzcd}
  \qquad\qquad
  \begin{tikzcd}[row sep={2em,between origins}, column sep={4.8em,between origins}]
   & M \ar[dd, "u" description] \ar[ld] \ar[rd] &  \\
  S  & \ar[l, phantom, pos=0.3, "\sigma\Downarrow" description] \ar[r, phantom, 
  pos=0.4, "\Downarrow\tau" description] & T , \\
  & \ar[ul] N \ar[ur] &
  \end{tikzcd}
  \]
required to be compatible in that they satisfy the two equations
  \begin{equation}\label{eq:two-eqs}
    \begin{tikzcd}[column sep={4em,between origins}, row sep={1.1em,between origins}]
  &  M \ar[ld, pos=0.40, "u"'] \ar[dddd]
  \\
   N \ar[dd]  \ar[rddd] & 
  \\[3pt]
  &
  \ar[ld, phantom, pos=0.81, "\delta'"]
  \ar[lu, phantom, pos=0.35, "\theta"]
  \\[3pt]
   S \ar[rd] &
  \\
  & \P 
  \end{tikzcd}
  =
  \begin{tikzcd}[column sep={4em,between origins}, row sep={1.1em,between origins}]
   & M \ar[ld, pos=0.40, "u"'] \ar[dddd]  \ar[lddd]
  \\
   N \ar[dd] &
  \\[3pt]
  &
  \ar[lu, phantom, pos=0.85, "\sigma"] 
  \ar[ld, phantom, pos=0.35, "\epsilon'"]
  \\[3pt]
   S \ar[rd] &
  \\
  & \P
  \end{tikzcd}
  \qquad\qquad
  \begin{tikzcd}[column sep={4em,between origins}, row sep={1.1em,between origins}]
  M \ar[rd, pos=0.40, "u"] \ar[dddd] & 
  \\
  & N \ar[dd]  \ar[lddd] 
  \\[3pt]
  \ar[ru, phantom, pos=0.35, "\theta" description] 
  \ar[rd, phantom, pos=0.81, "\delta" ] & 
  \\[3pt]
  & T \ar[ld]
  \\
  \P &
  \end{tikzcd}
  =
  \begin{tikzcd}[column sep={4em,between origins}, row sep={1.1em,between origins}]
  M \ar[rd, pos=0.40, "u"] \ar[dddd]  \ar[rddd] & 
  \\
  & N \ar[dd] 
  \\[3pt]
  \ar[ru, phantom, pos=0.85, "\tau" description] 
  \ar[rd, phantom, pos=0.35, "\epsilon" description] 
  & 
  \\[3pt]
  & T \ar[ld]
  \\
  \P &
  \end{tikzcd}
  \end{equation}
  (In the left-hand equation, the $2$-cells go left; in the right-hand equation the 
  $2$-cells go right, but in any case they are of course invertible.) 
  We see that $\theta$ is completely determined by $\sigma$, $\epsilon$, and
  $\delta$, and that $\theta$ is also completely determined by $\tau$,
  $\epsilon'$ and $\delta'$. In conclusion, the $\P$-span morphism is given by $u$
  together with the $2$-cells $\sigma$ and $\tau$, required to be compatible so as to
  induce the same $\theta$.
\end{blanko}

The general notion of morphism of $\P$-spans will only be used in some auxiliary 
constructions. Mostly we are only interested in the invertible morphisms:
\begin{blanko}{Equivalences of $\P$-spans.}
  An {\em equivalence} between $\P$-spans is a morphism $(u,\theta,\sigma,\tau)$ for which 
  $u$ is an equivalence. The following is a routine computation.
\end{blanko}

\begin{lemma}\label{lem:equivalent-spans}
  Equivalent spans define equivalent functors.
\end{lemma}

We can use the notion of $\P$-span equivalence to simplify calculations:
\begin{lemma}
  Every $\P$-span \eqref{eq:BZ2span} is equivalent to one where $\epsilon'$ is trivial
  (and also equivalent to one where $\epsilon$ is trivial). The equivalence can be 
  taken to fix the underlying span, so that only the parity structure on the apex 
  groupoid $M$ changes.
\end{lemma}
\begin{proof}
  The equivalent span is
  \[
  \begin{tikzcd}[row sep={2.5em,between origins}, column sep={5.6em,between origins}]
  & M \ar[dd] \ar[ld, "p"'] \ar[rd, "q"] &  \\
  S \ar[rd, "f"'] &
  \ar[l, phantom, pos=0.3, "\commutes" description]
  \ar[r, phantom, pos=0.36, "\footnotesize (\epsilon')^{-1}\epsilon" description]
  & T \ar[ld, "g"] \\
  & \P &
\end{tikzcd}
\]
   For the equivalence, define $(u,\theta)$ to be the triangle
    $$
  \begin{tikzcd}[column sep={2em,between origins}]
	M  \ar[rr, "\id"] \ar[rd, "m"'] & \ar[d, phantom, pos=0.3, "\epsilon'"] & N \ar[ld, 
    "f \circ p"] \\
	& \P  &
  \end{tikzcd}
  $$
  (with $\sigma$ and $\tau$ identities).
  The equations \eqref{eq:two-eqs} are immediately checked.
\end{proof}

As a consequence, the equivalence class of a span \eqref{eq:BZ2span} depends only on
the `ratio' between the two $2$-cells
$$
\rho := (\epsilon')^{-1} \epsilon .
$$

\begin{blanko}{Simplified spans.}
  It turns out to be more practical to define a span to be simply
\begin{equation}\label{eq:BZ2span-simplified}
\begin{tikzcd}[row sep={2em,between origins}, column sep={4.8em,between origins}]
  & M  \ar[ld, "p"'] \ar[rd, "q"] &
  \\
  S \ar[rd, "f"'] & \stackrel\rho\Rightarrow & T \ar[ld, "g"]
  \\
  & \P &
\end{tikzcd}
\end{equation}
because such simplified spans interact more elegantly with pullbacks.
Compared to the original span~\eqref{eq:BZ2span}, we have $\rho =
(\epsilon')^{-1} \cdot \epsilon$.

Spelled out, the natural transformation $\rho$ is thus the data of: for each point 
$x\in M$ an arrow $\rho_x$ in $\P$, such that for each arrow $a:x\to x'$ in $M$
we have the equation
$$
\begin{tikzcd}
pfx \ar[d, "pf(a)"'] \ar[r, "\rho_x"] & gqx \ar[d, "gq(a)"]  \\
pfx' \ar[r, "\rho_{x'}"'] & gqx'
\end{tikzcd}
$$
in the groupoid $\P$.

A simplified span is thus an equivalence class of fully specified spans 
as in \eqref{eq:BZ2span} where the equivalence relation is modulo equivalences of
$\P$-spans whose underlying equivalence of spans is the identity.
\end{blanko}

\begin{blanko}{Remark.}
  One reason for the utility of simplified spans and the convention of putting the
  $2$-cell to the right, is that the natural functor $\ZZ\to\linpm$ lands here: the
  functor sends $X{\to}\P$ to the slice $\ZZ_{/X}$ and it sends a morphism
    $$
  \begin{tikzcd}[column sep={2em,between origins}]
	X  \ar[rr] \ar[rd] & \ar[d, phantom, pos=0.3, "\rho"] & Y \ar[ld] \\
	& \P  &
  \end{tikzcd}
  $$
  to
  the functor $\ZZ_{/X}\to\ZZ_{/Y}$ given by
  $$ 
  \begin{tikzcd}[row sep={2em,between origins}, column sep={4.8em,between origins}]
      & X \ar[ld] \ar[dd] \ar[rd] &  \\
      X \ar[rd] &
      \ar[l, phantom, pos=0.3, "\commutes" description]
     \ar[r, phantom, pos=0.3, "\rho" description]
   & Y \ar[ld] \\
      & \P  ,  &
    \end{tikzcd}
  $$

  One way worry that in a simplified span the apex does not come 
  equipped with a parity structure, and therefore does not describe a well-defined linear 
  functor $q\lowershriek \circ p\upperstar : \ZZ_{/f}  \to  \ZZ_{/g}$. 
  However, a simplified span \eqref{eq:BZ2span-simplified} is equivalent to both
  \[
  \begin{tikzcd}[row sep={2em,between origins}, column sep={4.8em,between origins}]
   & M  \ar[ld, "p"'] \ar[rd, "q"] \ar[dd] &  \\
  S \ar[rd, "f"'] &
  \ar[l, phantom, pos=0.3, "\rho^{-1}" description] 
  \ar[r, phantom, pos=0.3, "\commutes" description] 
    & T \ar[ld, "g"] \\
  & \P &
  \end{tikzcd}
  \qquad
  \begin{tikzcd}[row sep={2em,between origins}, column sep={4.8em,between origins}]
   & M  \ar[ld, "p"'] \ar[rd, "q"] \ar[dd] &  \\
  S \ar[rd, "f"'] &
  \ar[l, phantom, pos=0.3, "\commutes" description] 
  \ar[r, phantom, pos=0.3, "\rho" description] 
    & T \ar[ld, "g"] \\
  & \P &
  \end{tikzcd}
  \]
  which have equivalent associated linear 
  functors,
  $$
  (q, \id)\lowershriek \circ (p,\rho^{-1})\upperstar
  \simeq
  (q, \rho)\lowershriek \circ (p, \id)\upperstar
  $$
  by Lemma~\ref{lem:equivalent-spans}.
\end{blanko}

\begin{blanko}{Composition of spans and the bicategory $\linpm$.}
  $\P$-spans are composed by pullback (which of course means homotopy
  pullback). Beck--Chevalley for pullback squares ensures that the linear
  functor defined by the pullback-composite of two $\P$-spans is canonically
  equivalent to the composition of the two individual functors associated
  to the two $\P$-spans.
  The identity span $S\stackrel=\leftarrow S \stackrel=\rightarrow S$
  (with the identity $2$-cell down to $\P$) defines the identity functor
  $\ZZ_{/f} \stackrel{\id}\rightarrow \ZZ_{/f}$.

  Altogether we have defined the bicategory $\linpm$: its objects are
  slices $\ZZ_{/f}$, and its morphisms are functors given by $\P$-spans (including those
  $2$-cells). The $2$-cells of the bicategory are equivalences of $\P$-spans.
  At the level of functors, these are the natural equivalences of linear
  functors.
  They do not play a big role in the following, and one could also get away
  with considering the $1$-category where the morphisms are equivalence
  classes of $\P$-spans.
\end{blanko}

\begin{blanko}{Remark.}
  One may simply say that $\linpm$ is the category whose objects are parity structures
  and whose morphisms are $\P$-spans, but the viewpoint of slice categories is
  preferred because the slice is the finite-homotopy-sum completion, and since one
  actually sees the linear functors (rather than just seeing matrix algebra).
  Furthermore, the slice viewpoint generalises elegantly to the case of an infinite
  base (slicing over an infinite groupoid) which is important in combinatorics. For
  example, the distinction (for an infinite set $S$) between the vector space $\Q_S$ and the
  profinite-dimensional vector space $\Q^S$ is very cleanly modelled by $\FF_{/S}$ and
  $\FF^{S}$ (which are {\em not} equivalent for infinite $S$). These points are made
  convincingly in the paper {\em Homotopy linear
  algebra}~\cite{Galvez-Kock-Tonks:1602.05082}.
\end{blanko}

\subsection{Monoidal structures}

Since the group $O(1)$ is abelian, the category $\P=BO(1)$ is monoidal (in 
fact symmetric monoidal): the 
monoidal structure
$$
\P \times \P \stackrel{\cdot}\longrightarrow \P
$$
is the unique assignment on points, and multiplication on arrows.

Since $\P$ is  monoidal, the slice $\ZZ=\grpd_{/\P}$ acquires a convolution 
tensor product:
$$
\begin{tikzcd}[row sep={2.6em,between origins}] S \ar[d, "f"'] \\ \P\end{tikzcd} 
  * 
\begin{tikzcd}[row sep={2.6em,between origins}] T \ar[d, "g"] \\ \P\end{tikzcd} 
 \ := \
  \begin{tikzcd}[row sep={2.5em,between origins}] 
	S\times T \ar[d, "f\times g"'] \\ \P\times\P 
	\ar[d, "\cdot"]\\ \P\end{tikzcd} 
$$
The unit for this monoidal structure is $\e: 1 \to \P$.
(Note that under the equivalence $\grpd_{/\P} \simeq 
\grpd^{\P}$, the monoidal 
structure $*$ corresponds to Day convolution.)

The canonical functor $\e\lowershriek : \grpd\simeq \grpd_{/1} \to \grpd_{/\P}$
is monoidal for the cartesian product on $\grpd$ and the convolution product on
$\grpd_{/\P}$, displayed for emphasis as
$$
(\FF,\times,1) \longrightarrow (\ZZ,\,*\,,\e) .
$$

\begin{blanko}{Tensor product.}
  The bicategory $\linpm$ now acquires a monoidal structure given by
  $$
  \ZZ_{/f} \atensor \ZZ_{/g}  := \ZZ_{/f*g} ,
  $$
  which is the objective analogue of the tensor product of vector spaces.
  The unit object for the monoidal structure is the slice $\ZZ_{/\e}$.
  It plays the role of the ground field. Note that $\ZZ_{/\e}$ is not equivalent 
  to $\ZZ$ (because $\e$ is not a terminal object).

  From its definition as a convolution product, it follows that $\ZZ_{/f} \atensor \ZZ_{/g}$
  can be characterised by a universal property, as the recipient of
  the universal bilinear functor.

  \bigskip

  We have a canonical monoidal functor
  $$
  (\lin,\tensor,\FF) \longrightarrow (\linpm,\atensor,\ZZ_{/\e}) ,
  $$
  which takes $\FF_{/S}$ to $\ZZ_{/S}$ where the second $S$ is really
  $S\to\P$ with the trivial parity structure (all arrows are even). This
  means that ordinary objective linear algebra over groupoids (as outlined in 
  \ref{sub:lin}) embeds into
  the new framework (and it will be compatible also with the notion of cardinality
  we introduce in Section~\ref{sec:card}).
\end{blanko}

\subsection{The negative of a span}

\label{sec:negative-span}

The {\em negative} of a $\P$-span 
\[
\begin{tikzcd}[row sep={2em,between origins}, column sep={4.8em,between origins}]
  & M \ar[ld] \ar[rd] &  \\
  S \ar[rd] & \epsilon & T \ar[ld] \\
  & \P &
\end{tikzcd}
\]
denoted $\overline M$, is by definition the same span of groupoids, but
with the opposite $2$-cell, denoted $\overline \epsilon$, so that for
each $x\in M$, we have
$$
\overline\epsilon _x = -1 \cdot \epsilon_x .
$$

In particular, for any $f:S\to \P$, consider the identity $\P$-span $S
\stackrel{=}\leftarrow S \stackrel{=}\to S$, with identity $2$-cell
$\epsilon_x= +1$ for all $x\in S$; then the negative $\overline S$ 
is the same span at the level of groupoids, 
but with $2$-cell $\overline \epsilon_x$ given by
$\overline\epsilon_x=-1$ for all $x\in S$. 

With reference to this $\overline S$,
the negative of $M$ can be obtained by precomposing
with $\overline S$:
$$
\overline M = \overline S \cdot M 
$$
(or by postcomposing with $\overline T$).

\subsection{Scalars}

\nocite{Kraus:1983}

\begin{blanko}{States, costates, scalars.}
  Recall that a {\em state} of an object $V$ in a monoidal category
  $(\CC,\tensor,\one)$ is a morphism $\one\to V$. A {\em costate} (or {\em 
  effect}~\cite{Kraus:1983}) is a
  morphism $V\to \one$, and a {\em scalar} is a endomorphism
  $\one\to\one$. In $(\vect,\tensor,\ground)$, an endomorphism 
  $\ground\to\ground$ is the same thing as just an element of the ground 
  field $\ground$, but in $(\linpm, \atensor,\ZZ_{/\e})$ there is a difference:
\end{blanko}

In $\linpm$,
the `ground field' is
$\ZZ_{/\e}$, and an element in this slice is a triangle
$$
\begin{tikzcd}[column sep={2em,between origins}]
  S  \ar[rr] \ar[rd] & \ar[d, phantom, pos=0.3, "\theta" 
  description] & 1 \ar[ld, "\e"] \\
  & \P     &
\end{tikzcd}
$$
(that is, an orientation of $S$).
In contrast:

\begin{blanko}{Scalars.}
  A {\em scalar} is a linear endofunctor $\ZZ_{/\e} \to \ZZ_{/\e}$, that is, a $\P$-span
\[
\begin{tikzcd}[row sep={2em,between origins}, column sep={4.8em,between origins}]
  & S \ar[ld] \ar[rd] &  \\
  1 \ar[rd, "\e"'] & \rho & 1   . \ar[ld, "\e"] \\
  & \P &
\end{tikzcd}
\]
We write $\widehat S$ to refer to the whole $\P$-span. In the fully-specified-span 
viewpoint, a scalar structure thus amounts to {\em two} orientations on a parity structure.

A {\em morphism of scalars} is just a morphisms of spans. (We will mostly be 
interested in the invertible morphisms of scalars.)
This defines the category of scalars $\scalars$.
\end{blanko}

\begin{blanko}{Important examples.}
  Here are three all-important scalars, to be thought of as $1$, $-1$, and a kind 
of $0$:
\[
\begin{tikzcd}[row sep={2em,between origins}, column sep={4.8em,between origins}]
  & 1 \ar[ld] \ar[rd] &  \\
  1 \ar[rd, "\e"'] & \commutes & 1  , \ar[ld, "\e"] \\
  & \P &
\end{tikzcd}
\qquad
\begin{tikzcd}[row sep={2em,between origins}, column sep={4.8em,between origins}]
  & 1 \ar[ld] \ar[rd] &  \\
  1 \ar[rd, "\e"'] & (-1) & 1  , \ar[ld, "\e"] \\
  & \P &
\end{tikzcd}
\qquad
\begin{tikzcd}[row sep={2em,between origins}, column sep={4.8em,between origins}]
  & 1\sqcup 1 \ar[ld] \ar[rd] &  \\
  1 \ar[rd, "\e"'] & \pi & 1   . \ar[ld, "\e"] \\
  & \P &
\end{tikzcd}
\]
The second scalar is the negative of the first (and they are not 
equivalent scalars).
The third scalar is the sum of the two previous ones: the natural transformation
$\pi$ is positive on one component and negative on the other. The importance of
this span is that the square diagram $\pi$ is precisely the homotopy pullback, or in
other words the loop space of $\P$, and we write
$$
1 \times_{\P} 1 = \Omega\P = 1 \sqcup 1 = \{ \pp,\mm \}  .
$$
\end{blanko}

\begin{blanko}{Sign splitting.}
  By the universal property of the pullback, the span
  $\widehat{\Omega\P}$ is the terminal scalar: for any other scalar
  $\widehat S$, there is induced a unique morphism of scalars
  $$
  S \to 1 \times_{\P} 1 = \{ \pp,\mm \}
  $$
  and thereby a sum splitting of the scalar $\widehat S$
  into two parts, a positive part and a negative part
  \begin{equation}\label{eq:sign-splitting}
  S = S_{\pp} \sqcup S_{\mm}  .
  \end{equation}
  When emphasis on this splitting is required, we shall write $\widehat{S}$ 
  as $\big( S_{\pp}, S_{\mm} \big)$.
\end{blanko}

In the fully-specified-span viewpoint (as in \eqref{eq:BZ2span}), the structure of
scalar on a parity groupoid can be interpreted as the data of {\em two} orientations.
The sign splitting then amounts to separating out the components where the
orientations agree and the components where the orientations do not agree. In the
simplified-span viewpoint \eqref{eq:BZ2span-simplified} often preferred, the two
orientations are not well defined individually, but their ratio is.

The sign splitting of a scalar is the first place where significant signs
appear. Unlike parities (and orientations), the sign splitting of scalars
as in \eqref{eq:sign-splitting} is a homotopy-invariant notion:

\begin{lemma}
  Given scalars $\widehat S=(S_\pp,S_\mm)$ and $\widehat T=(T_\pp,T_\mm)$, to give an 
  equivalence of scalars $\widehat S\simeq \widehat T$ is the same as giving two equivalences 
  of groupoids $S_{\pp}\isopil T_{\pp}$ and $S_{\mm}\isopil T_{\mm}$.
\end{lemma}

More precisely:
\begin{prop}\label{morphism-of-scalars}\label{prop:End=SS}
  The category of scalars $\scalars$ is equivalent to the category $\grpd_{/\Omega\P} =
  \grpd_{/\{\pp,\mm\}}$. This equivalence is monoidal: the monoidal structure on
  scalars is the convolution product $*$, and the monoidal structure on
  $\grpd_{/\{\pp,\mm\}}$ is also the convolution product (with respect to the
  monoidal structure on $\Omega \P$).
\end{prop}

\begin{blanko}{Remark.}
  A neat feature of the
  scalars-as-linear-endomorphisms-of-the-ground-field viewpoint is that
  multiplication of scalars becomes composition of linear maps.
  (Equivalently, it is the tensor product.) Spelling out the composition 
  of two
  scalars
  \[
  \begin{tikzcd}[row sep={2em,between origins}, column sep={3em,between origins}]
    && S \times S' \ar[ld] \ar[rd] &&
    \\
    & S \ar[ld] \ar[rd] & \commutes & S' \ar[ld] \ar[rd]
    \\
    1 \ar[rrdd] & \rho & 1 \ar[dd] & \rho' & 1 \ar[lldd] 
    \\
    &&&&
    \\
    && \P  &&
  \end{tikzcd}
  \]
  we find that the sign of an element $(s,s')$ is
  $$
  \rho_s \cdot \rho'_{s'}  ,
  $$
  where the $\cdot$ is composition of arrows in $\P$ (that is, the group 
  law of $O(1)=\{+1,-1\}$).
  An element $(s,s')$ is thus positive precisely when $s$ and $s'$ have the 
  same sign, and is negative precisely $s$ and $s'$ have opposite 
  signs. 
  This gives the formula
  $$
  \big( S_{\pp}, S_{\mm} \big) \cdot \big( S'_{\pp}, S'_{\mm} \big)
  = 
  \big( S_{\pp}\cdot S'_{\pp} + S_{\mm} \cdot S'_{\mm} \ ,
  \ S_{\pp}\cdot S'_{\mm} + S_{\mm} \cdot S'_{\pp} \big) ,
  $$
  which is the Brahmagupta sign rule, which is thus
  a consequence of the linear-algebra formalism (but of course ultimately a 
  consequence of that rule in $O(1)$).
\end{blanko}

\subsection{States (vectors)}

A {\em state} of an object $S$ in a monoidal category $(\CC,\tensor,\one)$ is
a linear map $\one \to S$. In many familiar concrete monoidal categories, this is 
essentially the same thing as an element of $S$. For example, in 
$(\kat{Set},\times,1)$, a state of a set $S$ is a morphism $1 \to S$, which is 
the same thing as an element in $S$, and in $(\vect, 
\tensor, \Q)$, a state $\Q \to V$ of a vector space $V$ is the same thing
as a vector in $V$. For our monoidal category $\linpm$, there is a 
difference, though. An object in $\ZZ_{/f}$ is
a triangle
$$
\begin{tikzcd}[column sep={2em,between origins}]
  X  \ar[rr] \ar[rd] & \ar[d, phantom, pos=0.3, "\theta" 
  description] & S \ar[ld, "f"] \\
  & \P     &
\end{tikzcd}
$$
(which we may think of as a bundle over $S$).
A state is more data:
\begin{blanko}{States (vectors).}
  A {\em state} of $\ZZ_{/S}$ is a linear functor $\ZZ_{/\e} \to 
\ZZ_{/S}$, meaning that it is a span 
  \[
\begin{tikzcd}[row sep={2em,between origins}, column sep={3em,between origins}]
  & X \ar[ld] \ar[rd] &  \\
  1 \ar[rd, "\e"'] & \rho & S  , \ar[ld, "f"] \\
  & \P &
\end{tikzcd}
\]
(in other words, it's a bundle together with an orientation ($\rho^{-1}$)).
\end{blanko}

\begin{blanko}{Elementary states.}
  The {\em elementary states} of $\ZZ_{/S}$ are given, for each
  $s \in S$, by the names:
    \[
\begin{tikzcd}[row sep={2em,between origins}, column sep={3em,between origins}]
  & 1 \ar[ld] \ar[rd, "\name{s}"] &  \\
  1 \ar[rd] & \commutes & S \ar[ld] \\
  & \P &
\end{tikzcd}
\]
Whenever we refer to an elementary state by writing 
$1{\leftarrow}1{\stackrel{\name{s}}\to} S$
it is understood that it is the positive version, with trivial $2$-cell down to
$\P$.
\end{blanko}

In ordinary objective linear algebra, where there are no signs, one 
chooses for each connected component an elementary state, given by
$\name s : 1 \to S$. Within a connected component, any two choices
are equivalent (as states).

In the signed situation, things are more complicated, as
every connected component can count with either sign, and the ambiguity is
between working with a basis vector $\basis_s$ or $-\basis_s$. So if there are 
$k$ connected components in the parity groupoid $S \to \P$, then there are
$2^k$ possible bases. The immediate ambiguity comes from the fact that for
each $s\in S$, there are two version of the name:
$$
\begin{tikzcd}[column sep={2em,between origins}]
  1  \ar[rr, "\name{s}"] \ar[rd, "\e"'] & \ar[d, phantom, pos=0.3, 
  "\text{\footnotesize $+1$}" 
  description] & S \ar[ld, "f"] \\
  & \P  &
\end{tikzcd}
\qquad
\begin{tikzcd}[column sep={2em,between origins}]
  1  \ar[rr, "\name{s}"] \ar[rd, "\e"'] & \ar[d, phantom, pos=0.3, 
  "\text{\footnotesize $-1$}" 
  description] & S \ar[ld, "f"] \\
  & \P  .  &
\end{tikzcd}
$$
(Considered as states, these two a generally not equivalent, by 
Lemma~\ref{not-eq} below.)
Between these two states, one may argue that the even version is more canonical 
than the odd version. But within a single component of $S$ there may be 
odd arrows $s\to t$, and then although $s$ and $t$ 
belong to the same connected component, the corresponding elementary 
states are of opposite signs:

\begin{lemma}
  If $\sigma:s\isopil t$ is an odd arrow in $S$, then the states
  \[
  \begin{tikzcd}[row sep={2em,between origins}, column sep={4.8em,between origins}]
	& 1 \ar[ld] \ar[rd, "\name{s}"] &  \\
	1 \ar[rd] & (+1) & S   , \ar[ld] \\
	& \P &
  \end{tikzcd}
  \qquad
  \begin{tikzcd}[row sep={2em,between origins}, column sep={4.8em,between origins}]
	& 1 \ar[ld] \ar[rd, "\name{t}"] &  \\
	1 \ar[rd] & (-1) & S   , \ar[ld] \\
	& \P &
  \end{tikzcd}
  \]
  are equivalent.
\end{lemma}
\begin{proof}
  An equivalence is given by
    \[
\begin{tikzcd}[row sep={2em,between origins}, column sep={4.8em,between origins}]
 & 1 \ar[dd] \ar[ld] \ar[rd, "\name{s}"] &  \\
1  &
\ar[l, phantom, pos=0.3, "\commutes" description]
\ar[r, phantom,  pos=0.3, "(\sigma)" description]
& S   . \\
& 1 \ar[lu] \ar[ru, "\name{t}"'] &
\end{tikzcd}
\]
It is routine to check the two equations in \eqref{eq:two-eqs}.
\end{proof}
\noindent
So in this case there is no way to say that $\name{s}$ is more canonical than 
$-\name{s}=\name{t}$.

\bigskip

The most striking instance of this is the case of an odd automorphism $\sigma: 
s\isopil s$:
\begin{cor}\label{plus=minus}
  Any odd automorphism $\sigma: s \isopil s$ induces an equivalence between
  the elementary state $1\leftarrow 1 \stackrel{\name{s}}\to S$ and its negative.
\end{cor}

This is a first hint at important cancellations taking place in connection with 
odd automorphisms.

Note also:
\begin{lemma}\label{not-eq}
  If $s\in S$ has no odd automorphisms, then the elementary state
  $1\leftarrow 1 \stackrel{\name{s}}\to S$ is not equivalent to its 
  negative.
\end{lemma}

\subsection{Group actions on $\P$-spans, states, and scalars}

We briefly point out some subtleties
regarding groups action on $\P$-spans, states and scalars, following up on
\ref{sub:group-act}.

\begin{blanko}{Actions on oriented parity structures.}
    For a $G$-action 
      $$
  \begin{tikzcd}[column sep={2em,between origins}]
    X  \ar[rr, ".g"] \ar[rd] &
    \ar[d, phantom, pos=0.3, "\theta_g" description]
    & X \ar[ld] \\
    & \P  &
  \end{tikzcd}
  $$
  to be compatible with an orientation
  \[
  \begin{tikzcd}[row sep={2em,between origins}, column sep={4em,between origins}]
   & X \ar[dd] \ar[ld]
   \\
  1 \ar[rd]& \ar[l, phantom, pos=0.3, "\omega" description]
  \\
  & \P
  \end{tikzcd}
  \]
  is to have the equation
  $$
  (\theta_g)_x \cdot \omega_{x.g} = \omega_x
  $$
  for each $x\in X $ and $g\in G$, as in the left-hand equation of \eqref{eq:two-eqs}. 
  This shows that for an orientation-preserving
  action, the $2$-cell data is forced: to give an orientation-preserving action, it is
  enough to give an action on the underlying groupoid $X$, then there is a unique
  $2$-cell making it into an action of oriented parity structures. In particular, if
  $X$ has the pure even orientation, then the $2$-cell data is trivial (that is,
  $\theta_g = \id$).
\end{blanko}

\begin{blanko}{Actions on $\P$-spans.}\label{act-on-spans}
  For $G$ a finite group, a $G$-action on a $\P$-span
\begin{equation}\label{eq:Gactspan}
\begin{tikzcd}[row sep={2em,between origins}, column sep={4.8em,between origins}]
	& X \ar[ld] \ar[dd] \ar[rd] &  \\
	S \ar[rd] &
    \ar[l, phantom, pos=0.3, "\commutes" description]
   \ar[r, phantom, pos=0.3, "\rho" description]
 & T \ar[ld] \\
	& \P  ,  &
  \end{tikzcd}
\end{equation}
is a functor from $BG$ to the category of $\P$-spans from $S$ to $T$ with value the 
given span. Unpacking the definition of morphism of $\P$-spans \ref{morphism-of-spans} we see that it amounts
to the data
    $$
  \begin{tikzcd}[column sep={2.5em,between origins}]
  X \ar[rd] \ar[rr, ".g"] & \ar[d, phantom, pos=0.25, 
  "\theta_g" description] & X \ar[ld]  \\
   & \P &
  \end{tikzcd}
\qquad    \qquad
  \begin{tikzcd}[row sep={2.3em,between origins}, column sep={4.5em,between origins}]
    &X \ar[dd, ".g"' description]\ar[ld]\ar[rd]& \\
   S  &
   \ar[l, phantom, pos=0.3, "\sigma_g" description]
   \ar[r, phantom, pos=0.3, "\tau_g" description]
   & T   \\
 & X \ar[lu]\ar[ru] & 
  \end{tikzcd}
  $$
  for each $g\in G$, compatible with the group law in $G$, and it should satisfy the 
  two conditions
  $$
      \begin{tikzcd}[column sep={4em,between origins}, row sep={1.1em,between origins}]
  &  M \ar[ld, pos=0.40, ".g"'] \ar[dddd]
  \\
   N \ar[dd]  \ar[rddd] & 
  \\[3pt]
  &
  \ar[ld, phantom, pos=0.81, "\commutes"]
  \ar[lu, phantom, pos=0.35, "\theta_g"]
  \\[3pt]
   S \ar[rd] &
  \\
  & \P 
  \end{tikzcd}
  =
  \begin{tikzcd}[column sep={4em,between origins}, row sep={1.1em,between origins}]
   & X \ar[ld, pos=0.40, ".g"'] \ar[dddd]  \ar[lddd]
  \\
   X \ar[dd] &
  \\[3pt]
  &
  \ar[lu, phantom, pos=0.85, "\sigma_g"] 
  \ar[ld, phantom, pos=0.35, "\commutes"]
  \\[3pt]
   S \ar[rd] &
  \\
  & \P
  \end{tikzcd}
  \qquad\qquad
  \begin{tikzcd}[column sep={4em,between origins}, row sep={1.1em,between origins}]
  X \ar[rd, pos=0.40, ".g"] \ar[dddd] & 
  \\
  & X \ar[dd]  \ar[lddd] 
  \\[3pt]
  \ar[ru, phantom, pos=0.35, "\theta_g" description] 
  \ar[rd, phantom, pos=0.81, "\epsilon" ] & 
  \\[3pt]
  & T \ar[ld]
  \\
  \P &
  \end{tikzcd}
  =
  \begin{tikzcd}[column sep={4em,between origins}, row sep={1.1em,between origins}]
  X \ar[rd, pos=0.40, ".g"] \ar[dddd]  \ar[rddd] & 
  \\
  & X \ar[dd] 
  \\[3pt]
  \ar[ru, phantom, pos=0.89, "\tau_g" description] 
  \ar[rd, phantom, pos=0.35, "\epsilon" description] 
  & 
  \\[3pt]
  & T \ar[ld]
  \\
  \P &
  \end{tikzcd}
$$
The left-hand equation gives 
$\theta_g = \pari(\sigma_g)$ (independently of $x$),
and with this the right-hand equation becomes
\begin{equation}\label{eq:span-act}
\pari((\sigma_g)_x) \cdot \rho_{x.g} = \rho_x \cdot \pari((\tau_g)_x) .
\end{equation}

The (weak) quotient of the $G$-action is
$$
\begin{tikzcd}[row sep={2em,between origins}, column sep={4.8em,between origins}]
	& X/G \ar[ld] \ar[dd] \ar[rd] &  \\
	S \ar[rd] &
    \ar[l, phantom, pos=0.3, "\commutes" description]
   \ar[r, phantom, pos=0.3, "\rho" description]
 & T \ar[ld] \\
	& \P  ,  &
  \end{tikzcd}
$$
It is calculated as the quotient $X/G$ in the category of groupoids, with the new legs
induced by the universal property. The $2$-cell structure remains the same pointwise: naturality
also with respect to the `new' arrows in $X/G$ follows from \eqref{eq:span-act}.

One should be aware that many quotients that exist in the category of groupoids do not 
extend to $\P$-spans because the action itself is not an action of $\P$-spans. Some examples 
are given below.
\end{blanko}

\begin{blanko}{Group actions on states.}\label{act-on-states}
  In the special case where the $\P$-span \eqref{eq:Gactspan}
  is a state, meaning that we have $S=1$,
  then necessarily $\sigma_g=\id$ (for all $g\in G$), and therefore also
  $\theta_g = \id$, and Equation~\eqref{eq:span-act} simplifies to
  \begin{equation}\label{eq:state-act}
  \rho_{x.g} = \rho_x \cdot \pari((\tau_g)_x) .
  \end{equation}
  This means that to give an action on a state $1\leftarrow X \to T$ is the same
  as giving an action on $X$ in $\ZZ_{/T}$ but required to have trivial $2$-cell
  data $\theta_g$.
\end{blanko}

\begin{blanko}{Group actions on scalars.}\label{act-on-scalars}
  Specialising to the case of a scalar $1\leftarrow S \to 1$,
  also $\tau_g$ becomes trivial, so we are left with just an action on
  the groupoid $S$, but it should still satisfy the condition, which now 
  boils down to
  $$
  \epsilon_{x.g} = \epsilon_x ,
  $$
  which is to say that the action must preserve the sign splitting 
  (as we already knew from
  Proposition~\ref{morphism-of-scalars}). So the group action splits into
  one group action on $S_\pp$ and one group action on $S_\mm$.
  
  The quotient for such an action is simply the scalar
  $$
  (S_\pp/G, S_\mm/G) .
  $$
\end{blanko}

The next few simple examples illustrate some of the subtleties with group actions
on states and scalars.
\begin{blanko}{Example.}\label{not-a-scalar-action}
  For a surjective group homomorphisms $f: G \to O(1)$ with kernel $K$
  and a group element $a\notin K$,
  consider the scalar 
    \begin{equation}
  \label{}
  \begin{tikzcd}[row sep={2em,between origins}, column sep={4.8em,between origins}]
	& \underline G \ar[ld] \ar[rd] &  \\
	1 \ar[rd] & \rho & 1 \ar[ld] \\
	& \P  ,  &
  \end{tikzcd}
  \end{equation}
  where $\rho_x$ is even for $x\in K$ and odd for $x \in K a$ (the nontrivial coset of
  $K$). So as a scalar, the set $\underline G$ has the sign splitting $(\underline 
  G_\pp, \underline G_\mm)
  = (K, K a)$. Now $G$ acts on the set $\underline G$ by right multiplication, but the
  odd group elements will interchange the positive and the negative. So it is not an
  action of scalars.
  
  Similarly, the quotient groupoid $\underline G/G=1$, although it is a scalar, is {\em not}
  the quotient as a scalar -- it does not even receive any morphism of scalars from 
  the original scalar.
\end{blanko}

\begin{blanko}{Example.}\label{non-action2}
  The state $1 \leftarrow 1 \to BG$ only admits one $G$-action (of states), namely the 
  trivial action with trivial $2$-cells. This follows from \ref{act-on-states}.
  Trying to define a non-trivial action on states by
    $$
  \begin{tikzcd}[column sep={2.5em,between origins}]
  1 \ar[rd] \ar[rr, ".g"] & \ar[d, phantom, pos=0.25, 
  "\commutes" description] & 1 \ar[ld]  \\
   & \P &
  \end{tikzcd}
\qquad    \qquad
  \begin{tikzcd}[row sep={2.1em,between origins}, column sep={4em,between origins}]
    &1 \ar[dd, ".g"']\ar[ld]\ar[rd]& \\
   1  &\ar[r, phantom, pos=0.3, "g" description]& BG   \\
 & 1 \ar[lu]\ar[ru] & 
  \end{tikzcd}
  $$
  does not work if $G$ has odd elements: while the left-hand equation of \eqref{eq:two-eqs} does hold,
  the right-hand equation does not (and it is not possible to modify the commutative 
  triangle on the left either, as that would break the left-hand equation).
\end{blanko}

\begin{blanko}{Example.}\label{false+false=true}
  In continuation of the previous two examples,
   consider the composite $\P$-span
     \[
  \begin{tikzcd}[row sep={2em,between origins}, column sep={3em,between origins}]
    && \underline G \times 1\ar[ld] \ar[rd] &&
    \\
    & \underline G \ar[ld] \ar[rd] & \commutes & 1 \ar[ld] \ar[rd]
    \\
    1 & \rho & 1  & \commutes & BG  
  \end{tikzcd} 
  \quad = \quad
  \begin{tikzcd}[row sep={2.7em,between origins}, column sep={4.0em,between origins}]
  & \underline G \ar[ld] \ar[rd] &   \\
  1 & \rho & BG
  \end{tikzcd}
  \]
  (Recall that $\rho_x$ is the parity of the group element $x\in\underline G$.) The
  two individual $\P$-spans each has its false $G$-action, as we have just seen, but
  the two faults cancel out so as to define a {\em true} action on the composite
  $\P$-span, namely the diagonal action
  $
  (x,1) \mapsto (x.g, 1.g) =(x.g,1)
  $, which is just
  $$
  x\mapsto x.g
  $$
  equipped with $2$-cells
  $$
  \begin{tikzcd}[column sep={2.5em,between origins}]
  \underline G \ar[rd] \ar[rr, "x{\mapsto}x.g"] & \ar[d, phantom, pos=0.25, 
  "\commutes" description] & \underline G \ar[ld]  \\
   & \P &
  \end{tikzcd}
\qquad    \qquad
  \begin{tikzcd}[row sep={2.1em,between origins}, column sep={4em,between origins}]
    &\underline G \ar[dd, ".g"']\ar[ld]\ar[rd]& \\
   1  &\ar[r, phantom, pos=0.3, "g" description]& BG   \\
 & \underline G \ar[lu]\ar[ru] & 
  \end{tikzcd}  
  $$
  This actually works: the required equation is (for every $x\in\underline G$ and 
  $g\in G$):
  $$
  \rho_{x.g} = \rho_x \cdot \pari(g) ,
  $$
  which clearly holds. Compared to the non-example \ref{not-a-scalar-action}, the fact
  that the span ends in $BG$ instead of ending in $1$ gives sufficient elbow room for
  the action to work out correctly.
\end{blanko}

The composition of the example is an instance of scalar multiplication, which we 
treat next:

\subsection{Scalar multiplication, inner product, and projections}
\label{sub:inner}

\begin{blanko}{Scalar multiplication.}
  As in any monoidal category, the scalars act on the states: given a scalar 
  $\widehat S$ and a state $\widehat X$
  \[
  \begin{tikzcd}[row sep={2em,between origins}, column sep={4.8em,between origins}]
	& S \ar[ld] \ar[rd] &  \\
	1 \ar[rd] & \sigma & 1  \ar[ld] \\
	& \P &
  \end{tikzcd}
  \qquad
  \begin{tikzcd}[row sep={2em,between origins}, column sep={4.8em,between origins}]
	& X \ar[ld] \ar[rd] &  \\
	1 \ar[rd] & \epsilon & J   , \ar[ld] \\
	& \P &
  \end{tikzcd}
  \]
  we have {\em scalar multiplication}
$$
\widehat S \cdot \widehat X
$$
defined either as the tensor product 
$S*X$, or as composition of spans, given as, respectively
  \[
\begin{tikzcd}[row sep={2.8em,between origins}, column sep={3.9em,between origins}]
  & S\times X \ar[ld] \ar[rd] & \\
  1\times 1 \ar[rd] & \sigma\times\epsilon & 1\times J \ar[ld]  \\
  & \P\times\P \ar[d, "\cdot"] & \\
  & \P &
\end{tikzcd}
\qquad = \qquad
\begin{tikzcd}[row sep={2.8em,between origins}, column sep={3.6em,between origins}]
  && S\times X \ar[ld] \ar[rd] && \\
  & S \ar[ld] \ar[rd] & \commutes & X \ar[ld] \ar[rd] &  \\
  1 
  \ar[rrd]
  & \phantom{xx}\sigma&1 
  \ar[d]
  &\epsilon\phantom{xx}& J
  \ar[lld]
  \\
  && \P  &&
\end{tikzcd}
\]
\end{blanko}

\begin{blanko}{Inner product.}
  Since our `vector spaces' $\ZZ_{/J}$ are essentially already 
coordinatised (not quite, because of the ambiguity of signs), 
they come equipped with inner products, or equivalently, a duality:
the {\em dual costate} to a state $1{\leftarrow}X{\to}J$ is simply the 
transposed span $J{\leftarrow}X{\to}1$ (with inverted $2$-cell, in the case of 
simplified spans). 

The (canonical) {\em inner product} $\innerprod{\widehat 
X}{\widehat Y}$
of two states
  \[
  \begin{tikzcd}[row sep={2em,between origins}, column sep={3.2em,between origins}]
	& X \ar[ld] \ar[rd] &  \\
	1  & \rho & J
  \end{tikzcd}
  \qquad
  \begin{tikzcd}[row sep={2em,between origins}, column sep={3.2em,between origins}]
	& Y \ar[ld] \ar[rd] &  \\
	1  & \sigma & J
  \end{tikzcd}
  \]
is by definition the scalar given by span composition (matrix multiplication)
of $\widehat X$ with the dual of $\widehat Y$:
\[
\innerprod{\widehat X}{\widehat Y} := \qquad
\begin{tikzcd}[row sep={2.4em,between origins}, column sep={3.8em,between origins}]
  && X \times_J Y \ar[ld] \ar[rd] && \\
  & X \ar[ld] \ar[rd] & \pi & Y \ar[ld] \ar[rd] & \\
  1 
  & \rho&J 
  &\sigma^{-1}& 1
\end{tikzcd}
\]

\end{blanko}

\begin{blanko}{Projection onto an elementary state.}
  In an inner-product space in ordinary linear algebra, the orthogonal projection 
  of a vector $\mathbf{x}$ onto a basis vector $\basis_j$ is given by the formula
  \begin{equation}\label{vect-proj}
  \operatorname{proj}_{\basis_j}(\mathbf{x}) = 
  \frac{\innerprod{\mathbf{x}}{\basis_j}}{\innerprod{\basis_j}{\basis_j}}
  \basis_j .
  \end{equation}
  This has a nice interpretation in homotopy linear algebra: for the 
  projection of the state $1 \leftarrow X \to J$ (for short, $\widehat X$)
  onto the elementary state $1 \leftarrow  1 \stackrel{\name j}\to J$, 
  we have
  \begin{equation}\label{wrong-act}
  \operatorname{proj}_j(\widehat X) = \frac{X_j}{j!} \cdot \name{j}  .
  \end{equation}
  
  This works with signs too, although the formula has to change slightly.
  In detail, the inputs are the $\P$-spans
    \[
  \begin{tikzcd}[row sep={2em,between origins}, column sep={3.2em,between origins}]
	& X \ar[ld] \ar[rd, "q"] &  \\
	1  & \rho & J
  \end{tikzcd}
  \qquad
  \begin{tikzcd}[row sep={2em,between origins}, column sep={3.2em,between origins}]
	& 1 \ar[ld] \ar[rd, "\name{j}"] &  \\
	1  & \commutes & J
  \end{tikzcd}
  \]
  This second state  we denote by $\widehat U$ (which makes sense if we put $U=1$).
  By definition of the inner product we have
  \begin{equation}\label{rho-pi}
    \innerprod{\widehat X}{\widehat U} \quad = \quad
  \begin{tikzcd}[row sep={2.4em,between origins}, column sep={3.8em,between origins}]
  && X\times_J U \ar[ld] \ar[rd] && \\
  & X \ar[ld] \ar[rd, "q"'] & \pi & U \ar[ld, "\name{j}"] \ar[rd] & \\
  1 
  & \rho &J 
  && 1
\end{tikzcd}
\end{equation}
  whose apex is just
  the fibre $X_j$,  and we also have
  $$
      \innerprod{\widehat U}{\widehat U} \quad = \quad
  \begin{tikzcd}[row sep={2.4em,between origins}, column sep={3.8em,between origins}]
  && U \times_J U \ar[ld] \ar[rd] && \\
  & U \ar[ld] \ar[rd, "\name{j}"'] & \pi & U \ar[ld, "\name{j}"] \ar[rd] & \\
  1 
  & &J 
  && 1
\end{tikzcd}
$$
whose apex is the loop space at $j \in J$, namely $\Omega B 
\Aut(j)$.
The idea is now that the division should be interpreted as the quotient by the canonical 
action of $j! = \Aut(j)$ on $X_j$ as in \eqref{wrong-act}. However, in the signed 
situation this does not quite work, because 
while $j!$ does act on the {\em groupoid} $X_j$,
it is not true in general that it acts on the {\em scalar} defined by $X_j$.
An example of this situation was given in \ref{not-a-scalar-action-2}.
However, it {\em is} true that $j!$ acts on the whole state $\widehat X_j \cdot \widehat U$,
which is given by
\begin{equation}\label{integrand/P}
\begin{tikzcd}[row sep={2.4em,between origins}, column sep={3.8em,between origins}]
  && X_j \ar[ld] \ar[rd] && \\
  & X_j \ar[ld] \ar[rd] &  & 1 \ar[ld] \ar[rd] & \\
  1 
  & \rho \cdot \pari(\pi)& 1
  &\commutes& J
\end{tikzcd}
\quad
= \quad
  \begin{tikzcd}[row sep={3em,between origins}, column sep={4.2em,between origins}]
   & X_j \ar[ld] \ar[rd] &  \\
  1 & \rho \cdot \pari(\pi) & J
  \end{tikzcd}
\end{equation}
The $2$-cell $\rho\cdot \pari(\pi)$ here is obtained by pasting $\rho$ and $\pi$ in 
the composed diagram \eqref{rho-pi}. Explicitly, recalling that the objects of $X_j$ are pairs
$(x,a)$
with $x\in X$ and $a:qx\isopil j$ in $J$, the component of $\rho\cdot \pari(\pi)$
on $(x,a)\in X_j$ is given by $\rho_x \cdot \pari(a)$.
The action is given by (for each $g\in j!$):
$$
  \begin{tikzcd}[column sep={2.5em,between origins}]
  X_j \ar[rd] \ar[rr, ".g"] & \ar[d, phantom, pos=0.25, 
  "\commutes" description] & X_j \ar[ld]  \\
   & \P &
  \end{tikzcd}
\qquad    \qquad
  \begin{tikzcd}[row sep={2.1em,between origins}, column sep={4em,between origins}]
    X_j \ar[dd, ".g"']\ar[rd]& \\
  \ar[r, phantom, pos=0.3, "\tau_g" description]& BG   \\
 X_j \ar[ru] & 
  \end{tikzcd}
$$ 
as in \ref{act-on-spans}--\ref{act-on-states}. The $2$-cell $\tau_g$ has component $g$
on every point $(x,a)$. At the groupoid level, the action is
$$
(x,a) \mapsto (x, a.g) .
$$
Equation~\eqref{eq:state-act} now reads
$$
\big( \rho\cdot \pari(\pi)\big)_{(x,a).g} = \big(\rho\cdot \pari(\pi)\big)_{(x,a)} 
\cdot \pari((\tau_g)_{(x,a)})
$$
which in turn unpacks to
$$
\rho_x \cdot \pari(a.g) = \rho_x\cdot\pari(a) \cdot \pari(g)
$$
which clearly holds.

So in the signed case,
the correct formula\footnote{We would like to suggest that all linear-algebra books 
in the universe be revised so as to write  $ 
  \frac{\innerprod{\mathbf{x}}{\basis_j}\cdot  \basis_j}{\innerprod{\basis_j}{\basis_j}}
 $ instead of $\frac{\innerprod{\mathbf{x}}{\basis_j}}{\innerprod{\basis_j}{\basis_j}}
  \basis_j$, but we are not 
completely sure whom to suggest it to.} for the projection is therefore
$$
\operatorname{proj}_j(\widehat X) =
\frac{X_j\cdot \name{j}}{j!} .
$$
\end{blanko}

\begin{blanko}{Example.}\label{not-a-scalar-action-2}
  For a surjective group homomorphism $f:G \to O(1)$, consider the state 
  \begin{equation}
  \label{}
  \begin{tikzcd}[row sep={2em,between origins}, column sep={4.8em,between origins}]
	& 1 \ar[ld] \ar[rd] &  \\
	1  & \commutes & BG 
  \end{tikzcd}
  \end{equation}
  Now the scalar given by the fibre is the scalar
  \begin{equation}
	\begin{tikzcd}[row sep={2em,between origins}, column sep={4.8em,between origins}]
	  && \underline G \ar[ld] \ar[rd] & \\
	  & 1 \ar[ld] \ar[rd] & \pi & 1 \ar[ld] \\
	  1 \ar[rd, "\e"'] & \commutes & BG \ar[ld, "f"] & \\
	  & \P & &
	\end{tikzcd}
  \end{equation}
  from Example~\ref{not-a-scalar-action}, and the corresponding elementary state
  is that of Example~\ref{non-action2}. The group action on the composed span is that of 
  Example~\ref{false+false=true}.
\end{blanko}

\begin{prop}
  Any state is the sum of its projections onto the elements of a basis. More 
  precisely:
  $$
1{\leftarrow}X{\to}J \ 
\simeq \ \sum_{j\in\pi_0(J)}  \frac{(1{\leftarrow}X_j{\to}1) \cdot 
(1{\leftarrow}1{\stackrel{\name{j}}\to}J)}{j!},
$$
or more conceptually:

  \[
  \begin{tikzcd}
   & X \ar[ld] \ar[rd] &  \\
  1 & \rho & J
  \end{tikzcd}
  \qquad 
  \simeq
  \qquad
  \int_{j\in J} \left(
  \begin{tikzcd}[row sep={1.8em,between origins}, column sep={2.8em,between origins}]
  && X_j \ar[ld] \ar[rd] && \\
  & X \ar[ld] \ar[rd] & \pi & 1 \ar[ld] \ar[rd] & \\
  1 
  & \rho&J 
  &\commutes& J
\end{tikzcd} \right)
\]
in the category of states of $J$. Or more compactly written:
$$
\widehat X \simeq \int_{j\in J} \name{j}\lowershriek \name{j}\upperstar (\widehat X) .
$$
\end{prop}

\begin{proof}
  We have already done most of the work, namely to establish that the natural action 
  on the integrand is an action of states. The rest can be checked at the level of 
  groupoids: the homotopy sum splits into a discrete
  sum over $\pi_0 J$ and then within each component a weak quotient of a $j!$-action,
  which at the groupoid level is the canonical
  action of $j!$ on $X_j$. Since the action exists at the level of states, also the
  quotient is a state, and that this quotient is $\widehat X$ can be measured at the 
  level of groupoids, where it is the well-known equivalence
  $X \simeq \sum_{j\in \pi_0 J} X_j/j!$.
\end{proof}

\subsection{Span composition in `coordinates'}

\begin{blanko}{Composition in `coordinates'.}
  To understand how pullback composition of $\P$-spans
  $$
  \begin{tikzcd}[row sep={2em,between origins}, column sep={3.6em,between origins}]
    && X\times_J Y \ar[ld] \ar[rd]&& \\
   & X \ar[ld] \ar[rd]&\pi& Y \ar[ld] \ar[rd]&  \\
  I &\rho & J & \rho'& K
  \end{tikzcd}
  $$
  corresponds to `matrix multiplication', we need the formula
  $$
  {}_i(X\times_J Y)_k \simeq \int_{j\in J} ( {}_i X_j \times {}_j Y_k) .
  $$
  From homotopy linear algebra without signs, we already have this equivalence at the
  level of groupoids, but we need the equivalence at the level of scalars,
  to ensure that the
  signs are preserved under it. This may look a bit worrying, since the composed span
  involves signs coming from three $2$-cells: $\rho$, $\rho'$ and $\pi$, whereas the
  two small spans only have signs coming from $\rho$ and $\rho'$, respectively.
  Nevertheless, it works out correctly. The formula will be a corollary of the
  following more general result with fewer indices.
\end{blanko}

\begin{prop}\label{prop:MjN}
   The equivalence of groupoids
   $$
   X\times_J Y \simeq \int_{j\in J} ( X_j \times {}_j Y) 
   $$
   is also an equivalence of $J$-spans from $X$ to $Y$.
\end{prop}
\begin{proof}
  We recall first the equivalence of groupoids.
  For fixed $j\in J$, consider the cube
  \[
  \begin{tikzcd}[row sep={2em,between origins}, column sep={4.8em,between origins}]
   & X_j \times {}_j Y \ar[ld] \ar[rd] \ar[dd, dotted] &   \\
  X_j \ar[dd] & & {}_j Y \ar[dd] \\
  & X \times_J Y \ar[ld] \ar[rd] & \\
  X \ar[rd] & \pi & Y \ar[ld] \\
  & J &
  \end{tikzcd}
  \quad
  =
  \quad
  \begin{tikzcd}[row sep={2em,between origins}, column sep={4.8em,between origins}]
   & X_j \times {}_j Y \ar[ld] \ar[rd]  &   \\
  X_j \ar[dd] \ar[rd] & \commutes & {}_j Y \ar[dd] \ar[ld] \\
  & 1 \ar[dd, pos=0.6, "\name{j}"']  & \\
  X \ar[rd] \ar[ru, phantom, "\sigma"] && Y \ar[ld] \ar[lu, phantom, "\tau"] \\
  & J &
  \end{tikzcd}
  \]
  The three squares $\pi$, $\sigma$, and $\tau$ are pullbacks.
  The dotted arrow is induced uniquely by the universal property 
  of the pullback (in the category of groupoids, so far), so as to make
  the unlabelled squares strictly commutative.
  For varying $j\in J$, the spans $X_j \times {}_j Y$ form a diagram of groupoids
  with colimit $X\times_J Y$.
  Since $J$ is a groupoid, a colimit over it splits into a discrete
  sum over $\pi_0 J$ and then within each component a weak quotient of a $j!$-action,
  in this case the diagonal 
  action of $j!$ on $X_j \times {}_j Y$.
  
  Our task is to check that the diagram is
  also a diagram of $J$-spans, and the key is to check that the 
  action is in fact an action of such spans.
    This happens because it is a
  diagonal action of two actions whose components cancel out. The objects in
  $X_j$ are pairs $(x,s)$ with $x \in X$ and $s: px\isopil j$. An
  element $g$ in the automorphism group $j!$ acts by postcomposition:
  $(x,s)\mapsto (x,s{\cdot} g)$. The objects in ${}_j Y$ are pairs
  $(y,t)$ with $y \in Y$ and $t: j \isopil qy$. An element $g$ in the
  automorphism group $j!$ acts by precomposition: $(y,t)\mapsto (y,g^{-1} {\cdot} t)$. 
  
  The action of $g\in j!$ on the groupoid $X_j \times {}_j Y$ is the diagonal action
  $$
  (x,s,y,t) \mapsto (x,s{\cdot} g,y,g^{-1}{\cdot} t)
  $$
  We claim this is furthermore an action of $J$-spans on the span
    \begin{equation}
  \begin{tikzcd}[row sep={2em,between origins}, column sep={4.8em,between origins}]
    & X_j \times {}_j Y  \ar[ld] \ar[dd] \ar[rd] &  \\
    X \ar[rd] & 
    \ar[l, phantom, pos=0.3, "\commutes"]
    \ar[r, phantom, pos=0.3, "\sigma\cdot \tau"] & Y \ar[ld] \\
    & J &
  \end{tikzcd}
  \end{equation}
  which for the sake of calculation we have arranged with the $2$-cell placed to the 
  right. The component of $\sigma\cdot \tau$ on a point $(x,s,y,t)$ is simply $s\cdot 
  t$.
  
  The $2$-cell data of this action (as in
  \ref{act-on-spans}, but over $J$ rather than over $\P$) are all identities:
    \[
      \begin{tikzcd}[column sep={4em,between origins}]
	X_j {\times} {}_j Y  \ar[rr, "g"] \ar[rd] & \ar[d, phantom, pos=0.3, "\commutes"] & 
    X_j {\times} {}_j \kern-1pt Y \ar[ld] \\
	& J \ar[d]  & \\
    & \P &
  \end{tikzcd}
\qquad
  \begin{tikzcd}[row sep={2em,between origins}, column sep={4.8em,between origins}]
   & X_j \times {}_j Y \ar[dd, "g" description] \ar[ld] \ar[rd] &  \\
  X  & \ar[l, phantom, pos=0.3, "\commutes" description] \ar[r, phantom, 
  pos=0.4, "\commutes" description] & Y , \\
  & \ar[ul] X_j \times {}_j Y \ar[ur] &
  \end{tikzcd}
  \]
  In the end, there is only one equation we need to check, namely
  $$
  (\sigma\tau)_{(x,s,y,t).g} = (\sigma\tau)_{(x,s,y,t)}.
  $$
  Since $g\in j!$ acts by $(x,s,y,t) \mapsto (x,sg,y,g^{-1}t)$, 
  the equation amounts to
  $$
  sg g^{-1} t = st
  $$
  which is certainly true in $J$.
   
  Having seen that the action is actually an action of spans over $J$, it follows 
  that the weak quotient at the level of groupoids is also the weak quotient in
  the category of spans, as required.
\end{proof}

\begin{blanko}{Remark.}
  It is worth noting that the natural action of $j!$ on the groupoid $X_j$ does not 
  give an action on 
  the $J$-span $X \leftarrow X_j \to 1$, and that
  natural action of $j!$ on the groupoid ${}_j Y$ does not give an action on the
  $J$-span $1 \leftarrow {}_j Y \to Y$.
  But the two false actions give together a true action on the composite $J$-span
    $$
  \begin{tikzcd}[row sep={2em,between origins}, column sep={3.6em,between origins}]
    && X_j\times {}_j Y \ar[ld] \ar[rd]&& \\
   & X_j \ar[ld] \ar[rd]&\commutes & {}_jY \ar[ld] \ar[rd]&  \\
  X & & 1 && Y
  \end{tikzcd}
  $$
\end{blanko}

The motivation for the above proposition was to show that 
composition of $\P$-spans can be interpreted as a kind o matrix multiplication
-- even at the objective level, before taking cardinality. (We will take cardinality
in the next section.) Let us record this:

\begin{cor}\label{cor:iMjNk}
  For composable $\P$-spans $\widehat X$ and $\widehat Y$ as above, we have a canonical
  equivalence of scalars
    $$
  {}_i(X\times_J Y)_k \simeq \int_{j\in J} ( {}_i X_j \times {}_j Y_k) .
  $$
\end{cor}

\begin{proof}
  Just take ${}_i( \ )_k$ fibre of the equivalence in Proposition~\ref{prop:MjN}, 
  noting that taking fibres is a pullback operation and therefore commutes with 
  colimits. The equivalence of $J$-spans from $X$ to $Y$ then becomes an equivalence 
  of $\P$-spans from $1$ to $1$, that is, an equivalence of scalars.
\end{proof}

\section{Cardinality}
\label{sec:card}

The strategy for defining the cardinality of a $\P$-span is to define first
the cardinality of scalars, since in this case it is clear what it should 
be (cf.~\eqref{normS} below). Then, for states and more general spans, the idea is to extract their
fibres, which are scalars, and define the cardinality in terms of the scalars.
This reflects the idea that we may think of a span as a `box full 
of groupoids', and then its cardinality should be the matrix whose entries are
the cardinalities of those groupoids. However, doing this naively would 
disregard the fact that all these groupoids are intertwined to make up the 
total apex of the span. This intertwining will involve both symmetries and signs.

\begin{blanko}{Example.}
  Given a state $1 \leftarrow X \to J$, 
  write it first as a homotopy linear combination of elementary states
  $$
  \widehat X \simeq \int_j X_j \cdot \name{j}
  $$
  where the fibres $X_j$ are scalars. Then define the cardinality to be
  $$
  \norm{\widehat X} = \Norm{\int_j X_j \cdot \name{j}} =
  \Norm{\sum_j \frac{X_j \cdot \name{j}}{j!}}
  \ := \ \sum_j \frac{\norm{X_j}}{\norm{j!}} \cdot \basis_j
  $$
\end{blanko}

\subsection{Cardinality of scalars}

We first define the {\em cardinality of a scalar} $\widehat S$ to be
\begin{equation}\label{normS}
\norm{\widehat S} := \norm{S_{\pp}} - \norm{S_{\mm}} .
\end{equation}

\begin{prop}\label{card-properties}
  Cardinality of scalars behaves as expected with respect to homotopy
  sums and products:
  \begin{align*}
    \norm{S+T} &= \norm{S} + \norm{T} \\
    \norm{S/G} &= \norm{S} / \norm{G} \\
    \norm{S * T} &= \norm{S} \cdot \norm{T} .
  \end{align*}
  
\end{prop}

\begin{proof}
  The checks are straightforward. In the case of the quotient, it must 
  be stressed that the action is required to be an action of scalars, 
  which in particular means that it preserves the sign splitting, as
  explained in \ref{act-on-scalars}. The action thus amounts to one action 
  on $S_\pp$ and one action on $S_\mm$, and the rest of the check is 
  easy.
\end{proof}

\begin{blanko}{Remark.}
  For any finite set $k$ we can form 
  $$
  S^k = \underset{k \text{ factors}}{\underbrace{S * \cdots * S}} .
  $$
  The compatibility of cardinality with $*$ now implies also
    $$
  \norm{S^k} = \norm{S}^{\norm{k}} .
  $$
\end{blanko}

\subsection{Odd automorphisms}

\begin{lemma}
  Consider the span composite
\[
\begin{tikzcd}[row sep={2.2em,between origins}, column sep={3.3em,between origins}]
  & & X_s \ar[ld] \ar[rd] & & \\
  & X \ar[ld] \ar[rd, "p"'] && 1 \ar[ld, pos=0.4, "\name{s}"'] \ar[rd] &  \\
  1 & \epsilon & S &\commutes & 1 \end{tikzcd}
\]
  If $s$ has an odd automorphism, then the scalar $\widehat X_s$ has cardinality zero.
    \marginpar{\dbend}
\end{lemma}

\begin{proof}
  The elements of the fibre $X_s$ are pairs $(x,a)$, with $x\in X$ and $a: px\isopil 
  s$ an arrow in $S$. We get a splitting of the scalar via
  $X_s 
  \dashrightarrow 1 \sqcup 1 = 
  \{\pp,\mm\}$ so as to get 
  $$
  X_s = (X_s)_\pp  \sqcup (X_s)_\mm  \,.
  $$
  The sign of a point $(x,a)$ is $\epsilon_x \cdot \pari(a)$.
  The automorphism group $s!$ acts naturally on the groupoid $X_s$
  by
  $$
  (x,a) \mapsto (x,a.g).
  $$
  We have already seen a couple of times
  that this action is not an action of scalars, since for odd $g\in s!$ it will not preserve 
  the sign splitting. On the contrary: it is clear that an odd $g\in s!$ will
  send positive points to negative points and negative points to positive points.
  As such it defines a sign-reversing involution on $X_s$, an isomorphism of
  groupoids 
  $$
  (X_s)_\pp \simeq (X_s)_\mm 
  $$
  making it clear that the total cardinality of the scalar is zero.
\end{proof}

\subsection{Cardinality as a functor}

\begin{blanko}{Discussion.}
  Given a $\P$-span $I \leftarrow M \to J$, we first extract all the two-sided
  fibres ${}_i M_j$. Each is naturally a scalar, and has cardinality
  $$
  \norm{{}_i M_{j}} = \norm{({}_i M_{j})_{\pp}} - \norm{({}_i M_{j})_{\mm}}  .
  $$
  The cardinality of the whole span is going to be a matrix of numbers.
  The entries of this numerical matrix are {\em not} just $\norm{{}_i M_j}$ but
  rather $$\norm{{}_i M_j}/\norm{j!},$$
  as is already the case in homotopy
  linear algebra without signs~\cite{Galvez-Kock-Tonks:1602.05082}, as already 
  discussed.  
  
  There are two problems with this definition: one is that throwing in 
  a symmetry factor should not happen unless it already has a meaning objectively
  as a quotient of a group action, and we have already seen examples illustrating that
  this $j!$-action is not always an action on scalars. The second problem is that in
  the presence of odd automorphisms of a point $j\in J$, the corresponding coefficients
  will always have cardinality zero, and is therefore not relevant as a basis element.
  The solution to both problems is to exclude all non-orientable points in the 
  definition of cardinality.
\end{blanko}

\begin{blanko}{Definition: the cardinality functor $\linpm \to \vect$.}
  \label{card-slice} 
  We define the cardinality functor, first on objects as \marginpar{\dbend}
  \begin{eqnarray*}
    \norm{ \ - \ } : \linpm & \longrightarrow & \vect  \\
    \ZZ_{/T} & \longmapsto & \Q_{\pi_0 T^\circ} ,
  \end{eqnarray*}
  the vector space spanned by symbols $\basis_t$, for $t$ running in a fixed choice of
  an elementary 
  state $1\stackrel{t}\to T$ for each orientable component of $T$.
  
  To define the cardinality of a morphism we need to 
  assign to each  $\P$-span
  $I \leftarrow M \to J$  a linear map 
$$
\Q_{\pi_0 I^\circ} \to \Q_{\pi_0 J^\circ}  ,
$$
and since these vector spaces are  already equipped with bases, we can specify it
with a matrix $[M]$.
We define $[M]$ to be the $\pi_0(I^\circ)$-by-$\pi_0(J^\circ)$ 
matrix whose $(i,j)$-entry is
$$
{}_i[M]_j := \norm{{}_i M_j/j!} .
$$
Note that the $j!$-action on ${}_i M_j$ is a bona fide action on scalars since $j$
does not have odd automorphisms.

It remains to check functoriality, which we state as a proposition:
\end{blanko}

\begin{prop}
  The cardinality assignment is functorial: the cardinality of a composed
  $\P$-span 
  \[
\begin{tikzcd}[row sep={2.5em,between origins}, column sep={3em,between origins}]
  && M \times_J N \ar[ld] \ar[rd] &&
  \\
  & M  \ar[ld] \ar[rd] & \pi & N \ar[ld] \ar[rd] &
  \\
  I 
  & \rho & J
  & \rho' & K
\end{tikzcd}
\]
  is the matrix product of the two individual cardinality matrices.
  Precisely, for fixed $i\in \pi_0 (I^\circ)$ and $k\in \pi_0 (K^\circ)$ we have
  the following identity of rational numbers:
  $$
  {}_i[M{\times\!_J} N]_k = \sum_{j \in \pi_0(J^\circ)} {}_i [M]_j \cdot {}_j [N]_k
  $$
\end{prop}

\begin{proof}
  We have already done most of the work. From Corollary~\ref{cor:iMjNk} we have an
  equivalence of scalars 
  $$
  {}_i(M\times_J N)_k \simeq \int_{j\in J} ( {}_i M_j \times {}_j N_k) .
  $$
  Now take cardinality of this equivalence, and use compatibility of
  cardinality with sums, quotients and products (cf.~\ref{card-properties}) to move
  the cardinality signs further in, so as to arrive at
$$
\norm{{}_i (M{\times\!_J} N)_{k}}
=
\sum_{j\in \pi_0 J}  \frac{\norm{{}_i M_{j}} \cdot \norm{{}_j N_{k}}}{\norm{j!}} .
$$
Finally divide by $\norm{k!}$  on both sides of the equation to arrive
at 
$$
\norm{{}_i (M{\times\!_J} N)_{k}/k!}
=
\sum_{j\in \pi_0 J}  \norm{{}_i M_{j} /j!} \cdot \norm{{}_j N_{k}/k!} ,
$$
which is what we wanted to show.
\end{proof}

\begin{blanko}{Philosophical remarks.}
  It is a pleasing feature that the signs do not appear in these proofs. They are 
  taken care of by the formalism and the attention to scalars rather than to signs.
  
  One might wonder if the non-orientable points, the odd automorphisms, could have been
  excluded from the beginning, already in Section 2. But this would not be possible,
  because odd automorphisms arise naturally in many constructions, 
  even in problems that start with no automorphisms at all. 
  The final section will illustrate that. Allowing odd automorphisms
  is necessary for an elegant theory, even
  if they are not visible in cardinality.
\end{blanko}

\begin{blanko}{Remark on calculation.}\label{iMj}
  The two-sided fibres are computed by
  pullback like this:
  \[
  \begin{tikzcd}[row sep={2.5em,between origins}, column sep={3em,between origins}]
    &&& {}_i M_{j} \ar[ld] \ar[rd] &&& \\
    && {}_i M \ar[ld] \ar[rd] && M_{j} \ar[ld] \ar[rd] && \\
    & 1  \ar[rd, "\name i"'] && M \ar[ld, "s"] \ar[rd, "t"'] && 1 \ar[ld, 
    "\name j"] & \\
     && I \ar[rd] && J \ar[ld] && \\
    &&& \P .  &&&
  \end{tikzcd}
  \]
  All the squares here carry a $2$-cell: the bottom
  square is the original $\P$-span structure of $M$, and the other three squares are
  homotopy pullbacks. The scalar structure on the two-sided fibre involves the
  composed $2$-cell.

  For the purpose of calculation, it is practical to assume (or arrange)
  that $s$ is a fibration.  Then we can calculate the two-sided
  fibres ${}_i M_j$ with the help of only one homotopy pullback and two
  strict ones:
  \[
  \begin{tikzcd}[row sep={3em,between origins}, column sep={4em,between origins}]
    &&& {}_i M_{j} \ar[ld] \ar[rd] &&& \\
    && {}_i M \ar[ld] \ar[rd] & \strict & M_{j} \ar[ld, onto] \ar[rd] && \\
    & 1  \ar[rd, "\name i"'] & \strict & M \ar[ld, onto, "s"] \ar[rd, "t"'] & \ho & 1 \ar[ld, 
    "\name j"] & \\
     && I \ar[rd] &\epsilon & J \ar[ld] && \\
     &&& \P &&&
  \end{tikzcd}
  \]
  The point is that thanks to the homotopy pullback, also the map $M \leftarrow 
  M_j$ becomes a fibration, so that the two strict pullbacks are still 
  homotopically meaningful. In this stricter (but homotopy equivalent) 
  version of the two-sided fibre, 
  the objects of ${}_i M_j$ are pairs $(m,\gamma)$ where $s(m)=i$ and 
  $\gamma: t(m) 
  \isopil j$, and the arrows from $(m,\gamma)$ to $(m',\gamma')$ are 
  $\theta: m \isopil m'$ in $M$ such that $s(\theta)=\id_i$ and
  \[
  \begin{tikzcd}[row sep={2em,between origins}, column sep={4.2em,between origins}]
  t(m) \ar[dd, "t(\theta)"']  \ar[rd, pos=0.35, "\gamma"] &  \\
   &  j , \\
  t(m') \ar[ru, pos=0.35, "\gamma'"'] &
  \end{tikzcd}
  \]
  In the scalar ${}_i M_j$, the sign of a point $(m,\gamma)$ is now the
  product $\epsilon_m \cdot \pari(\gamma)$.

  It is no real loss of generality to  demand $s$ to be a fibration,
  because it could always be replaced by a fibration, for example by
  homotopy pullback along the identity map. We shall exploit this
  asymmetric, stricter set-up in the calculation of determinants in
  Section~\ref{sec:det}.
\end{blanko}

\subsection{The universal property of the cardinality functor}

We have the natural functor $\P \to \vect$ sending the unique point to $\Q$ and sending
the arrows $\pm1$ to the linear maps $\pm \Id: \Q\to\Q$ (which are just scalars).
In complete analogy, we also have the functor $\P \to \linpm$, sending the unique point to $\ZZ_{/e}$ and
sending the arrows $\pm1$ to the linear functors (again just scalars)

\begin{equation}\label{pm1}
\begin{tikzcd}[row sep={2em,between origins}, column sep={4.8em,between origins}]
  & 1 \ar[ld] \ar[rd] &  \\
  1 \ar[rd, "\e"'] & (+1) & 1  , \ar[ld, "\e"] \\
  & \P &
\end{tikzcd}
\qquad
\begin{tikzcd}[row sep={2em,between origins}, column sep={4.8em,between origins}]
  & 1 \ar[ld] \ar[rd] &  \\
  1 \ar[rd, "\e"'] & (-1) & 1  , \ar[ld, "\e"] \\
  & \P &
\end{tikzcd}
\end{equation}
respectively. Closely related is the functor $\P \to \ZZ$, sending the unique point
to $\e: 1 \to \P$, which, as we have seen, exhibits $\ZZ$ as the finite-homotopy-sum completion
of $\P$. Here we are dealing with a linear version of that, and proceed to pinpoint 
the universal property, both of $\P\to\linpm$ and of the cardinality functor.

\begin{prop}\label{prop:universal}
  The cardinality functor is the unique finite-homotopy-sum preserving functor making 
  this triangle commute (up to natural isomorphism):
  \[
\begin{tikzcd}[row sep={3.6em,between origins}, column sep={4.2em,between origins}]
  & \P \ar[ld] \ar[rd] &
  \\
  \linpm \ar[rr, dashed, "\norm{ \; - \; }"'] && \vect
  \end{tikzcd}
\]
\end{prop}
\begin{proof}
  Commutativity of the triangle forces us to take $\norm{\ZZ_{/\e}} = \Q$.  
  Next we consider the case of $\ZZ_{/BG}$, 
  for $BG{\to}\P$ corresponding to a group homomorphism $G 
  \to O(1)$.
  We already know (from \ref{ex:BG-as-colimit})
  that $BG{\to}\P$ is (uniquely) the weak quotient of the $G$-action on $\e$
  given by $g\in G$ acting as
  \[
  \begin{tikzcd}[row sep={3.8em,between origins}, column sep={3.8em,between origins}]
  1 \ar[rd] \ar[rr, "\id"] & \ar[d, phantom, pos=0.25, 
  "\sign(g)" description] & 1 \ar[ld]  \\
   & \P &
  \end{tikzcd}
  \]
  This colimit description in $\ZZ$ is also a colimit description in $\linpm$.
  But these maps are in the image of the functor $\P\to\linpm$, so we are forced to
send them to $\pm\Id : \Q\to\Q$. Since the functor we are constructing has to
preserve this quotient description of $BG{\to}\P$, 
the value on $BG{\to}\P$ therefore has to be the quotient
in $\vect$ of the sign action of $G$ on $\Q$. (Note that the quotient in $\vect$ is
just the ordinary vector-space quotient, because $\vect$ is just a $1$-category.)
But as soon as $G$ has an odd
element, this quotient is the zero vector space! 
In other words $\ZZ_{/BG}$ is sent to the vector space $\{0\} = \Q_\emptyset = 
\Q_{\pi_0(BG)^\circ}$, if $BG$ has odd automorphisms.
Finally for the general case:
For any object
  $\ZZ_{/S} \in\linpm$ (with reference to a parity structure 
  $S{\to}\P$), we first write $S$ as a homotopy sum of copies of $\e$ 
  (cf.~\ref{lem:colim-of-e}).
  This expands to a discrete sum over $s\in \pi_0(S)$ of objects
  $B\Aut(s) \to \P$ (with parity structure induced from $S$). 
  For each of these summands we get a contribution $\Q_{\pi_0BG^\circ}$,
  and since the functor is required to preserve sums, we get altogether
$$
\norm{\ZZ_{/S}} = \Q_{\pi_0S^\circ} ,
$$
forced by the requirement of compatibility with the functors from $\P$.

This shows that the assignment on objects is forced. For the assignment on
morphisms we first do the case of a scalar.
  We already know where the two basic scalars \eqref{pm1} are sent. Any other scalar
  is a homotopy sum of a number of copies of these basic scalars, so there is no
  other choice than the assignment already made in the definition of cardinality.
  
  We omit the general case.
  Instead, we discuss the broader context of the situation.
\end{proof}

There is a broader context of the universal property pointed out to us
by Maxime Ramzi. The following general result was proved by 
Harpaz, motivated by chromatic homotopy theory.

\begin{theorem}[{Harpaz~\cite[Thm.4.1]{Harpaz:1703.09764}}]
  The $\infty$-category of spans of $m$-finite $\infty$-groupoids is the
  free $m$-semiadditive $\infty$-category on a point. This means that for any
  $m$-semiadditive $\infty$-category $\DD$, the natural functor
  $$
  \Fun^{m\operatorname{-sum}}(\spand_m, \DD) \longrightarrow \DD
  $$
  given by evaluation at $\mathcal{S}_m$, is an equivalence.
\end{theorem}
We briefly explain the notions involved, referring to Harpaz's paper for details.
An $\infty$-groupoid is $m$-finite if it is finite and $m$-truncated.
So $1$-finite $\infty$-groupoids are the groupoids of the present paper.
The colimits required preserved by the functors are colimits over $m$-finite 
$\infty$-groupoids, so in the case $m=1$ it is about what in the present paper is called
homotopy sums. $\spand_m$ is the category of spans of $m$-finite groupoids, 
so for $m=1$ it is what we call $\lin$ (and for $m=\infty$ it is the 
$\infty$-category $\lin$ from \cite{Galvez-Kock-Tonks:1602.05082}).

Explaining the notion of $m$-semiadditivity (originally due to Hopkins and Lurie, see
\cite{Harpaz:1703.09764}) would take us too far afield,
as it is a subtle notion,
requiring an inductive definition. Roughly an $\infty$-category is $m$-semiadditive if colimits
over $m$-finite $\infty$-groupoids are identified with limits, along a specific
comparison functor whose definition relies on the assumption that the category is
already $(m-1)$-semiadditive.

The $m=1$ instance of Harpaz's theorem thus reads
\begin{cor}
  The category $\lin$ is the
  free $1$-semiadditive category on a point. This means that for any
  $1$-semiadditive category $\DD$, the natural functor
  $$
  \Fun^{\operatorname{hsum}}(\lin, \DD) \longrightarrow \DD
  $$
  given by evaluation at $\FF$, is an equivalence.
\end{cor}

We have quoted the above theorem because the idea of it could be explained in a few 
paragraphs. But it is actually a more general theorem (also due to 
Harpaz) that we need:

\begin{theorem}[{Harpaz~\cite[Thm.5.29]{Harpaz:1703.09764}}]
  The $\infty$-category of $\CC$-decorated spans of $m$-finite $\infty$-groupoids is the
  free $m$-semiadditive $\infty$-category on $\CC$.
\end{theorem}

We do not wish to spell out the details of the statement, nor the definition of 
$\CC$-decorated span, but only note that when $\CC$ is itself an $m$-groupoid, then
$\CC$-decorated span just means span in the slice category over $\CC$.
The theorem can now be applied to the case of the present paper, where $m=1$ and
$\CC=\P$, and we get the following result, which is thus a special case of
Harpaz's general theorem:

\begin{prop}
  The category $\linpm$ is the free $1$-semiadditive category on $\P$. It means that
  for any $1$-semiadditive category $\DD$ with a functor $\P\to \DD$,
  there is a unique finite-homotopy-sum preserving functor
  \[
\begin{tikzcd}[row sep={3.6em,between origins}, column sep={4.2em,between origins}]
  & \P \ar[ld] \ar[rd] &
  \\
  \linpm \ar[rr, dashed] && \DD
  \end{tikzcd}
\]
\end{prop}

Finally, the category $\vect$ of rational vector spaces is $1$-semiadditive
(in fact $m$-semiadditive for all $m\leq \infty$),
and our Proposition~\ref{prop:universal} can be seen as just identifying
what the universal functor $\linpm\to\vect$ is: it is the cardinality functor
introduced in this section.

\section{Exterior powers and determinants}

\label{sec:det}

Throughout we denote by $\Sigma_k$ the symmetric group on $k$ letters.

\subsection{Exterior powers in $\ZZ$}

For $\Sigma_k$ acting on a groupoid $S$, when we say `quotient' it  
always means the 
weak quotient:, 
the objects of the groupoid $S/\Sigma_k$ are the 
objects of $S$, and an arrow in the quotient groupoid 
$S/\Sigma_k$ is by definition a pair $(a:x\to y, \sigma)$ 
which then has 
domain $x$ and codomain $y.\sigma$. 

Recall that a $\Sigma_k$-action on a parity structure $f:S \to 
\P$
is a functor
$$
B\Sigma_k \to \ZZ 
$$
sending the unique point to $f:S\to\P$. This means that
for each $\sigma\in \Sigma_k$ there is given
a morphism of parity structures
\[
\begin{tikzcd}[column sep={2.5em,between origins}]
S \ar[rd] \ar[rr, "\sigma"] & \ar[d, phantom, pos=0.25, 
"\epsilon_\sigma" description] & S \ar[ld]  \\
 & \P &
\end{tikzcd}
\]
with the $2$-cell $\epsilon_\sigma$ being part of the data, and of course it is 
required that $\epsilon_{\sigma\tau} = \epsilon_\sigma \cdot \epsilon_\tau$.
Since $\ZZ$ has finite colimits, we can form the (weak) quotient, 
denoted $S/\Sigma_k$. Its underlying groupoid will always be $S/\Sigma_k$
(since the forgetful functor $\ZZ\to\FF$ preserves colimits).

\begin{blanko}{Example.}
  If $f:S\to\P$ is a parity structure, and if $\Sigma_k$ acts on the underlying groupoid $S$,
  then there are two canonical ways to turn this action into an action on parity
  groupoids: {\em either} we can define $\epsilon$ to be constant $+1$ (that's the
  pure even action), {\em or} we can define $\epsilon$ to be always the sign of the
  permutation $\sigma$ (independently of $s\in S$), with reference to 
  the sign homomorphism $p: \Sigma_k \to O(1)$. This second action, 
  called the {\em signed action},
  will be our focus:
  \[
  \begin{tikzcd}[column sep={2.7em,between origins}, row sep={3.2em,between origins}]
  S \ar[rd] \ar[rr, "\tilde\sigma"] & \ar[d, phantom, pos=0.25, "p(\sigma)" 
  description] & S \ar[ld]  \\
   & \P &
  \end{tikzcd}
  \]
  (We decorate $\sigma$ with a tilde to stress that the action is 
  signed.)
\end{blanko}

\begin{blanko}{The exterior power in $\ZZ$.}
  Presently we are mainly interested in one particular $\Sigma_k$-action:
  namely, for a parity groupoid $X$ we consider the $\Sigma_k$-action on
  $X * \cdots * X$ given by permuting the $k$ factors, and where the
  $2$-cell $\epsilon$ is the signed action. Note that the underlying
  groupoid of $X* \cdots * X$ is the cartesian product $X\times \cdots
  \times X$ which it is useful to write also as $\Map(k,X)$.

  For $f:X \to \P$ in $\ZZ$ and $k\in \N$, we define the $k$th {\em exterior 
  power} 
  \marginpar{\dbend}
  $$
  \Hak k X := X^k/\Sigma_k
  $$
  with structure map given by $X^k/\Sigma_k \to \P^k/\Sigma_k \to \P$.
  Crucially, the last map involves the parity map $p:\Sigma_k \to O(1)$,
  which is where the alternating signs come from. In detail, a morphism
  in $X^k/\Sigma_k$ from $(y_1,\ldots,y_k)$ to $(x_1,\ldots,x_k)$
  consists of a permutation $\sigma \in \Sigma_k$ and $k$ morphisms $\gamma_j
  \colon y_j \to x_{\sigma(j)}$ in $X$. Its parity is $p(\sigma) f(\gamma_1)
  \cdots f(\gamma_k) \in O(1)$.

  It looks like we are defining the symmetric power, not the exterior
  power, but we shall quickly see that it is a reasonable definition. A
  first reassuring observation is that it has the correct `dimension':  
\end{blanko}
  
\begin{blanko}{`Dimension' of $\Hak k X$.}
  If $X$ has $n$ (orientable) components, and if we use the parity
  groupoid $\Hak k X$ to form the slice $\ZZ_{/\Hak k X}$ as a model for
  a vector space, then according to \ref{card-slice}, that vector space is
  $\Q_{\pi_0(\Hak k X)^\circ}$, spanned by the orientable components of
  $\Hak k X$. But among the maps $\phi \in \Map(k,X)/k!$, all
  non-injective maps have an odd automorphism: if two elements in $k$ are
  mapped to the same point in $X$ (or just the same component), then the
  transposition of those two elements will constitute an odd automorphism
  of $\phi$. Conversely, if all elements are mapped to distinct connected
  components, then only the trivial permutation can fix $\phi$. So the
  orientable part of $\Map(k,X)/k!$ is $\Inj(k,X^\circ)/k!$. Note that
  the action of $\Sigma_k$ on $\Inj(k,X^\circ)$ is free, so the groupoid
  $(\Hak k X)^\circ = \Inj(k,X^\circ)/k!$ has ${n \choose k}$ components,
  as expected from a notion of exterior power.
\end{blanko}

\begin{blanko}{Remark.}
  It is of course the intention that the elements of the groupoid $\Hak k
  X$ should be denoted $x_1 \wedge \cdots \wedge x_k$, but in the present
  paper, where we are still concerned with setting up the theory, we find
  it more practical to stick with tuple notation $(x_1,\ldots,x_k)$.
\end{blanko}

\begin{blanko}{Example: Exterior powers of a discrete set.}\label{discrete-ext}
  If $X$ is just a discrete set of cardinality $n$, then for each $k\in \N$ the exterior 
  power $\Hak k X$ is a parity groupoid (not a discrete one), but its 
  orientable part is (homotopy) discrete. Indeed, the orientable points are tuples 
  $\underline x = (x_1,\ldots,x_k)$ of {\em distinct} points, so the 
  automorphism group of $\underline x$ is trivial.
\end{blanko}

\begin{prop}
  The exterior power $\Hak k X$ is the weak quotient of the signed permutation action 
  of $\Sigma_k$ on $X^k$, in other words the
  (homotopy) colimit of the corresponding functor
  $B\Sigma_k \to \ZZ$.
\end{prop}

(For contrast, the (homotopy) colimit of $B\Sigma_k \to \ZZ$
corresponding to the straight permutation action of $\Sigma_k$ on $X^k$
is the ordinary weak quotient $X^k/k!$ with the pure-even parity
structure. This is the {\em symmetric power} of $X$.)

\bigskip

The following proof is not very different from the standard proof that
the weak quotient is the homotopy colimit in the category of groupoids.
We give the details because we are keen on being very careful with the
parity structures.

\begin{proof}
  First of all we check that $\Hak k X$ is a cocone under the action.
  At the level of underlying groupoids, it is clear that we have the required 
  $2$-cells
  \[
  \begin{tikzcd}[column sep={4.8em,between origins}, row sep={2.7em,between origins}]
    X^k  \ar[dd, "\tilde\sigma"'] \ar[rd, "q"]  & \\
    \ar[r, phantom, pos=0.4, "p(\sigma)" description] & \Hak k X \\
    X^k \ar[ru, "q"'] &
  \end{tikzcd}
  \]
  because the triangle does not commute, but the permutation preventing it from 
  commuting is precisely available as a morphism in $\Hak k X$.
  It remains to check that this is over $\P$: for this we should first specify that 
  the quotient map $q: X^k \to \Hak k X$ comes with trivial $2$-cell. 
  The required equation is now
  \[
  \begin{tikzcd}[column sep={3.0em,between origins}, row sep={2.7em,between origins}]
    X^k  \ar[dd, "\tilde\sigma"'] \ar[rr, "q"]  && \Hak k X \ar[dd] \\
	 & \ar[lu, phantom, pos=0.4, "(\sigma)" description]\ar[rd, phantom, pos=0.4, 
	 "\commutes" description]& \\
    X^k \ar[rruu] \ar[rr] && \P
  \end{tikzcd}
  \quad = \quad
  \begin{tikzcd}[column sep={3.0em,between origins}, row sep={2.7em,between origins}]
    X^k  \ar[dd, "\tilde\sigma"'] \ar[rr, "q"] \ar[rrdd] && \Hak k X \ar[dd] \\
	 & 
	 \ar[ld, phantom, pos=0.4, "p(\sigma)" description]\ar[ru, phantom, pos=0.4, 
	 "\commutes" description]
	 & \\
    X^k  \ar[rr] && \P
  \end{tikzcd}
  \]
  which is clear.  Consider now a challenger cocone under the action, so a 
  \[
\begin{tikzcd}[column sep={2.5em,between origins}]
X^k \ar[rd] \ar[rr, "r"] & \ar[d, phantom, pos=0.35, "\rho" description] & C \ar[ld]  \\
 & \P &
\end{tikzcd}
\]
  with a
  compatible system of
  \[
  \begin{tikzcd}[column sep={4.8em,between origins}, row sep={2.7em,between origins}]
    X^k  \ar[dd, "\tilde\sigma"'] \ar[rd, "r"]  & \\
    \ar[r, phantom, pos=0.4, "\beta_\sigma" description] & C \\
    X^k \ar[ru, "r"'] &
  \end{tikzcd}
  \]
  This has to be over $\P$, of course, meaning that we have 
  \[
  \begin{tikzcd}[column sep={3.0em,between origins}, row sep={2.7em,between origins}]
    X^k  \ar[dd, "\tilde\sigma"'] \ar[rr, "r"]  && C \ar[dd] \\
	 & \ar[lu, phantom, pos=0.4, "\beta_\sigma" description]\ar[rd, phantom, pos=0.4, 
	 "\rho" description]& \\
    X^k \ar[rruu] \ar[rr] && \P
  \end{tikzcd}
  \quad = \quad
  \begin{tikzcd}[column sep={3.0em,between origins}, row sep={2.7em,between origins}]
    X^k  \ar[dd, "\tilde\sigma"'] \ar[rr, "r"] \ar[rrdd] && C \ar[dd] \\
	 & 
	 \ar[ld, phantom, pos=0.4, "p(\sigma)" description]\ar[ru, phantom, pos=0.4, 
	 "\rho" description]
	 & \\
    X^k  \ar[rr] && \P
  \end{tikzcd}
  \]  
  The statement is now that there is a unique $u: \Hak k X \to C$ filling
  \begin{equation}\label{unique-u}
    \begin{tikzcd}
      X^k \ar[rrd, "r", bend left] \ar[rd, "q"] \ar[dd, "\tilde\sigma"']  &&
      \\
      \ar[r, phantom, "p(\sigma)" description]& \Hak k X \ar[r, dotted, "u"]  & C 
      \\
      X^k \ar[rru, "r"', bend right] \ar[ru, "q"']& 
    \end{tikzcd}
  \end{equation}
  to give $\beta_\sigma$. We can arrange for the two triangles to commute
  strictly, and there is a truly unique $u$ to achieve this: on objects the
  values of $u$ are forced because $q$ is the identity on objects. For the same
  reason the value of $u$ on `old' arrows is forced. For the `new' arrows, observe
  that each appears precisely once as a component of the natural
  transformation $p(\sigma)$, so the $u$-value on such a new arrow
  is therefore forced to be the corresponding
  component of $\beta_\sigma$. The $2$-cell
    \[
  \begin{tikzcd}[column sep={2.5em,between origins}]
  \Hak k X \ar[rd] \ar[rr, "u"] & \ar[d, phantom, pos=0.35, "?" description] & C 
  \ar[ld]  \\
   & \P &
  \end{tikzcd}
  \]
  `hanging under' $u$ is uniquely determined by $\rho$. If we allow the
  two triangles in \eqref{unique-u} to be not strictly commutative but
  carry instead a $2$-cell, then there is some slack in the definition of
  $u$, as we can shift such $2$-cells from one triangle to the other, but
  the different choices of $u$ are related by uniquely defined $2$-cells,
  so the space of choices is still contractible. (In the details of this
  argument, it is again important that $q$ is the identity on objects.)
\end{proof}

\begin{blanko}{Definition.}
  A cocone 
  \[
  \begin{tikzcd}[column sep={2.5em,between origins}]
  X^k \ar[rd] \ar[rr, "r"] & \ar[d, phantom, pos=0.35, "\rho" description] & C \ar[ld]  \\
   & \P &
  \end{tikzcd}
  \]
  like in the proof is called an {\em alternating map}.
\end{blanko}
 The description of $\Hak k X$ as a colimit now reformulates to:
\begin{cor}
  The quotient map $q: X^k \to \Hak k X$ is the universal alternating map.
\end{cor}

\begin{blanko}{Remark.}
  For vector spaces, a key property of the exterior powers is that if
  there is a repeated factor, then the exterior power is zero, say
  $v\wedge v = 0$ (and sometimes this is taken as the definition of
  exterior power). At the objective level, `zero' is not too reliable a
  concept, but it is clear that $(x,x) \in \Hak 2 X$ is a non-orientable
  point, as it has an odd automorphism. So while not exactly `zero', it
  does disappear after taking cardinality.
\end{blanko}

\subsection{Exterior powers in $\linpm$}

The goal of this subsection is to show that the universal property of 
$\Hak k X$ in $\ZZ$ extends to a universal property of 
$$
\Hak k (\ZZ_{/X}) :=
\ZZ_{/\Hak k X}
$$
in $\linpm$. In other words, the universal property holds also for
alternating maps given by spans (rather than just the pure-forward maps
in $\ZZ$ it was designed for). We first need to define alternating linear
functors.

\begin{blanko}{Alternating linear functors.}
  A linear functor $\ZZ_{/X} \atensor \cdots \atensor \ZZ_{/X} \to \ZZ_{/W}$
  is called {\em alternating} when it is equipped with, for each $\sigma\in 
  \Sigma_k$,
  a
  compatible system of $2$-cells
  \[
  \begin{tikzcd}[column sep={8em,between origins}, row sep={2.7em,between origins}]
    \ZZ_{/X} \atensor \cdots \atensor \ZZ_{/X}  \ar[dd, 
    "\tilde\sigma\lowershriek"'] \ar[rd]  & \\
    \ar[r, phantom, pos=0.4, "\alpha_\sigma" description] & \ZZ_{/W} \\
    \ZZ_{/X} \atensor \cdots \atensor \ZZ_{/X}  \ar[ru] &
  \end{tikzcd}
  \]
  This has to be over $\P$, of course. The compatibility is just 
  $\alpha_\sigma \alpha_t = \alpha_{\sigma \tau}$.
  
  At the level of spans, this means that we have a span
  $$
  X^k \stackrel{s}\longleftarrow A \stackrel{t}\longrightarrow W
  $$
  and compatible $2$-cells
  \[
  \begin{tikzcd}[column sep={4.8em,between origins}, row sep={2.7em,between origins}]
    X^k  \ar[dd, "\tilde\sigma"'] & \ar[l, "s"'] A \ar[rd, "t"] &
    \\
    \ar[rr, phantom, pos=0.42, "\alpha_\sigma" description] && W \\
    X^k  & \ar[l, "s"] A \ar[ru, "t"'] &
  \end{tikzcd}
  \]
  By the way composition of spans is defined, we can arrange this as a 
  pullback square and a triangle $2$-cell
  \[
  \begin{tikzcd}[column sep={5.2em,between origins}, row 
    sep={2.7em,between origins}]\label{tilde-sigma}
    X^k  \ar[dd, "\tilde\sigma"'] & \ar[l, "s"'] A \ddlpullback 
    \ar[dd, "\tilde\sigma"'] \ar[rd, "t"] &
    \\
    & \ar[r, phantom, pos=0.4, "\alpha_\sigma" description] & W \\
    X^k  & \ar[l, "s"] A \ar[ru, "t"'] &
  \end{tikzcd}
  \]
  The totally of the maps $A\to A$ constitute a $\Sigma_k$-action on $A$.
  These maps have been denoted by $\tilde\sigma$ again, because the action is the 
  pullback along $s$ of the signed action on $X^k$. Note also that the data $(t,\alpha_\sigma)$ 
  exhibits $W$ as a cocone under this action.

  Having defined alternating linear maps, we are ready to state the 
  universal property:
\end{blanko}

\begin{prop}
  The quotient map
  $$
  q\lowershriek : \ZZ_{/X} \atensor\cdots\atensor \ZZ_{/X} 
  \longrightarrow \Hak k (\ZZ_{/X})
  $$
  is the universal alternating linear map. That is, any other 
  alternating linear map $\ZZ_{/X} \atensor\cdots\atensor \ZZ_{/X}  
  \longrightarrow \ZZ_{/W}$
  factors uniquely as
  \[
  \begin{tikzcd}
  \ZZ_{/X} \atensor\cdots\atensor \ZZ_{/X} \ar[d, "q\lowershriek"'] 
  \ar[rr] && \ZZ_{/W}  \\
    \Hak k (\ZZ_{/X})  \ar[rru, dotted, "u"'] && 
  \end{tikzcd}
  \]
  (by a linear map $u$).
\end{prop}

\begin{proof}
  All the input data are maps inside $\ZZ$, although we have suppressed
  the structure maps to $\P$ in the notation. The construction exploits
  only universal properties holding in $\ZZ$, so it is not necessary to
  check any compatibility over $\P$. Given a span (that is, a $\P$-span)
  $X^k \stackrel s \leftarrow A \stackrel t \rightarrow W$ we need to
  exhibit (uniquely) the data
  \begin{equation}\label{barA}
  \begin{tikzcd}[column sep={5.2em,between origins}, row sep={2.7em,between origins}]
    X^k  \ar[dd, "q"'] & \ar[l, "s"'] A \ddlpullback 
    \ar[dd, dotted] \ar[rd, "t"] &
    \\
    & \ar[r, phantom, pos=0.4, "\alpha_\sigma" description] & W \\
    \Hak k X  & \ar[l, dotted, "\bar s"] \bar A \ar[ru, dotted, "\bar t"'] &
  \end{tikzcd}
  \end{equation}
  We put $\bar A := A/\Sigma_k$, with reference to the action induced by
  pullback from the signed permutation action on $X^k$ as in the
  definition of alternating functor~\eqref{tilde-sigma}. It means that
  $\bar A$ is the geometric realisation of a simplicial groupoid,
  $\Sigma_k{}^\bullet \times A$, just like $\Hak k X$ is the geometric
  realisation of the simplicial groupoid $\Sigma_k{}^\bullet \times X^k$,
  and the two fit together into a diagram
  \[
  \begin{tikzcd}[column sep={10em,between origins}, row sep={5em,between origins}]
    \ar[d, phantom, "\vdots" description]& \ar[d, phantom, "\vdots" 
    description] \\
    \Sigma_k{} \times \Sigma_k{} \times X^k \ar[d] \ar[d, shift left=2, "\text{act\phantom{j}}"]\ar[d, shift 
  right=2, "\text{proj}"'] & \Sigma_k{} \times \Sigma_k{} \times A 
  \dlpullback
  \ar[l, "\id\times\id\times s"] \ar[d] \ar[d, shift left=2, "\text{act\phantom{j}}"]\ar[d, shift 
  right=2, "\text{proj}"']\\
  \Sigma_k{} \times X^k \ar[d, shift left, "\text{act\phantom{j}}"]\ar[d, shift 
  right, "\text{proj}"'] & 
  \Sigma_k{} \times A \dlpullback \ar[d, shift left, "\text{act\phantom{j}}"]\ar[d, shift 
  right, "\text{proj}"'] \ar[l, "\id\times s"']  \\
  X^k \ar[d, dotted, "q"'] & 
  A \ar[d, dotted] \ar[l, "s"']  \\
  \Hak k X & \bar A \ar[l, dotted, "\bar s"]
  \end{tikzcd}
  \]
  All the squares in the simplicial map are pullback squares, by 
  construction of the action on $A$. We invoke descent to conclude that also
  the square to the geometric 
  realisations are pullbacks: by Thomason's 
  hocolimit theorem~\cite{Thomason:hocolims}, the weak quotient is also 
  an $\infty$-categorical colimit, so it follows from descent 
  \cite[Thm.6.1.0.6]{Lurie:HTT} in the $\infty$-topos of spaces over $\P$ that the 
  square with the colimits is a pullback too, since all the squares in the 
  simplicial map are.
  We have now constructed the 
  pullback square in \eqref{barA}. The triangle follows from the 
  universal property of $\bar A$: since the notion of alternating functor
  involved the condition that $W$ is a cocone under the 
  $\Sigma_k$-action on $A$, the universal property of the quotient $\bar 
  A$ delivers the map $\bar t$ together with the triangle $2$-cells 
  $\alpha_\sigma$, and this choice is unique.
  In fact also the choice of $\bar A$ was forced: the 
  square in \eqref{barA} has to be a pullback, and since by the 
  definition of alternating linear functor also all the squares in the 
  simplicial map are pullbacks, it follows from universality of colimits 
  (the fact that $\ZZ$ is locally cartesian closed) that $\bar A$ has to 
  be the colimit.
\end{proof}

\begin{blanko}{Functoriality of exterior powers.}
  The exterior power of a $\P$-span is just given component-wise: for
  \[
  \begin{tikzcd}[row sep={2em,between origins}]
    & M \ar[dd, phantom, "\epsilon" description] \ar[ld] \ar[rd] &  \\
    S \ar[rd] && T \ar[ld] \\
    & \P &
  \end{tikzcd}
  \]
  the $k$th exterior power is
  \begin{equation}\label{k-power-def}
  \begin{tikzcd}[row sep={2em,between origins}]
     & M^k/\Sigma_k \ar[dd, phantom, "\epsilon^k/\Sigma_k" description] \ar[ld] \ar[rd] &  \\
    S^k/\Sigma_k \ar[rd] && T^k/\Sigma_k \ar[ld] \\
    & \P^k/\Sigma_k \ar[dd] & \\
    &&\\
    & \P &
  \end{tikzcd}
  \end{equation}
  where the last parity map $\P^k/\Sigma_k$ is induced by $p:\Sigma_k\to 
  O(1)$.
  This is compatible with composition of spans, because the functor
  $\Grpd \to \Grpd$, $X \mapsto X^k/\Sigma_k$ preserves pullbacks
  (cf.~\cite{Gepner-Haugseng-Kock:1712.06469}).
\end{blanko}

\subsection{Determinants}

We will define the determinant of an endo-span as the top exterior power 
(with respect to its domain). 

\begin{blanko}{Set-up and notation.}
  Throughout this subsection, we fix a groupoid $X$ with $n$ orientable 
  components. We fix a basepoint $x_i$ in each orientable component, and form
  the tuple
  $$
  \underline x = (x_1,\ldots,x_n) .
  $$
\end{blanko}

\begin{blanko}{Top exterior powers.}
  With notation as above, the {\em top exterior power} $\Hak n X$ has as
  elements $n$-tuples of elements in $X$. As soon as any entry in the 
  $n$-tuple has an odd automorphism, then the whole $n$-tuple will have an
  odd automorphism in $\Hak n X$, and if there is a pair of isomorphic 
  elements in the $n$-tuple, an odd automorphism is given by 
  transposition. It follows that the orientable points in $\Hak n X$ are
  precisely the $n$-tuples of orientable points that are furthermore
  pairwise non-isomorphic.

  The $n$-tuple $\underline x=(x_1,\ldots,x_n)$ is an orientable point in
  $\Hak n X$, and every orientable point is isomorphic to it. The
  description of $\Hak n X$ as a quotient of a $\Sigma_n$-action gives a
  precise description of these isomorphisms: If $(y_1,\ldots,y_n) \simeq
  (x_1,\ldots,x_n)$ is an isomorphism in $\Hak n X$, then since the
  individual entries in these tuples are pairwise non-isomorphic, there
  is a uniquely determined permutation $\sigma \in \Sigma_n$ such that
  $$
  (y_1,\ldots,y_n) 
  \simeq (x_{\sigma 1},\ldots,x_{\sigma n}) 
  \simeq (x_1,\ldots,x_n)
  $$
  where the first part is a tuple of maps in $X$, say $\gamma_i : y_i
  \isopil x_{\sigma i}$, and the second part is the permutation 
  $\sigma^{-1}$.
  
  In the special case where $\underline y = \underline x$, we see that
  the automorphism group of $\underline x$ is the product of the
  automorphism groups of the $n$ points (which are all even by
  assumption). In summary:
\end{blanko}

\begin{lemma}\label{x!}
  Let $X$ be a groupoid with $n$ orientable components, with base points
  $x_1,\ldots,x_n$, respectively. Then the orientable part of the {\em top exterior
  power} $\Hak n X$ is connected and has automorphism group $\underline 
  x!=x_1!\cdots
  x_n!$.
\end{lemma}
\noindent
(Note that if $X$ has no orientable components, then the top exterior 
power is $\Hak 0 X = X^0/0! = 1$.)

\begin{blanko}{Remark.}
  We call it the top exterior power because it is the highest power that 
  has an orientable point. But note that the higher powers are not empty 
  -- they are just not orientable.
\end{blanko}

\begin{blanko}{The determinant.}
  Suppose $X$ is a finite groupoid with $n$ orientable components as
  above. The {\em determinant} of an endo-span $X \leftarrow A \to X$,
  denoted $\Det(\widehat{A})$, is by definition the span
  $$
  \Hak n X \leftarrow \Hak n A \to \Hak n X .
  $$
  (The $2$-cell structure is that of \eqref{k-power-def}.)
\end{blanko}

The fact that all tuples are involved, not just injective tuples,
means that the determinant is a much bigger object (cf.~\ref{immaterial} below) than what might be 
expected from classical linear algebra. In particular it is not just a scalar or a single groupoid.
But since $(\Hak n X)^\circ$ is connected (with basepoint
$\underline x = (x_1,\ldots,x_n)$), the scalar given by the two-sided fibre
$$
{}_{\underline x} (\Hak n A)_{\underline x}
$$
is the only entry in the orientable part of the span, and the 
cardinality of this $1$-by-$1$ matrix is
$$
\norm{\Det \widehat A}/\norm{\underline x!} ,
$$
a rational number. We proceed to show that this rational number agrees
with the ordinary determinant of the cardinality matrix of $\widehat A$.
This will be Corollary~\ref{cor:Det} below.
We actually prove an objective version of the result:

\begin{prop}[Objective Leibniz expansion]\label{Det=det}
  Let $X \stackrel s \leftarrow A \stackrel t \to X$ be an endo-span in 
  $\ZZ$, and let $\underline x = (x_1,\ldots,x_n)$ be a tuple of 
  orientable points, one for each orientable component of $X$. Then 
  there is a canonical equivalences of scalars
  $$
  {}_{\underline x} (\Hak n A)_{\underline x} \simeq 
  \sum_{\sigma\in\Sigma_n} \sign(\sigma) \prod_{i\in n} {}_{x_i} 
  A_{x_{\sigma i}}  .
  $$  
\end{prop}

\begin{cor}\label{cor:Det}
  For any endo-span $X \stackrel s \leftarrow A \stackrel t \to X$ in 
  $\ZZ$, we have an equality of rational numbers
  $$
  \norm{\Det \widehat{A}}  = \det \norm{\widehat{A}}.
  $$
\end{cor}

The Corollary will follow by taking cardinality of the 
equivalence from Proposition~\ref{Det=det} (and tracking a few symmetry 
factors). 
  
\begin{proof}[Proof of the Proposition.]
  The two-sided fibre   $$
  D := {}_{\underline x} (\Hak n A)_{\underline x}
  $$
  is the limit $D$ of the solid diagram
  $$
  \begin{tikzcd}[column sep={4.0em,between origins},row sep={3.5em,between origins}]
  && D \ar[ld, dotted] \ar[rd, dotted] &&  \\
  & P \ar[ld, dotted] \ar[rd, dotted] && Q \ar[ld, dotted] \ar[rd, dotted] & \\
  1 \ar[rd, "\name{\underline x}"']&& \Hak n A \ar[ld, "s"] \ar[rd, "t"'] && 1 \ar[ld, "\name{\underline x}"] \\
  & \Hak n X \ar[rd] && \Hak n X \ar[ld]& \\
  && \P  .  &&
  \end{tikzcd}
  $$
  For the convenience of the computation of this limit, we
  assume that $s$ is a fibration. (This is no loss of generality, as it
  could always be arranged by a fibrant replacement without changing the homotopy 
  types.)
  We can now use the description of the two-sided fibres given in
  Example~\ref{iMj}. The objects of $D$ are pairs $(\underline a,
  \gamma)$ where $\underline a \in \Hak n A$ is such that $s(\underline
  a) = \underline x$, together with an isomorphism $\gamma : t(\underline
  a) \simeq \underline x$ in $\Hak n X$. A morphism from $(\underline a,
  \gamma )$ to $(\underline a', \gamma ')$ is a morphism $\theta:
  \underline a \sim \underline a'$ in $\Hak n A$ with $s(\theta) =
  \id_{\underline x}$ and such that
    \[
    \begin{tikzcd}[column sep={4.8em,between origins}, row sep={2.7em,between origins}]
      t(\underline a)  \ar[dd, "t(\theta)"'] \ar[rd, "\gamma "]  & \\
       & \underline x \\
      t(\underline a') \ar[ru, "\gamma '"'] &
    \end{tikzcd}
    \]
  The condition $s(\theta) = 
  \id_{\underline x}$ implies that such a morphism $\theta$ cannot involve 
  permutation.

  The isomorphism $\gamma : t(\underline a) \isopil \underline x$ has a well-defined
  permutation component $\sigma$ because the entries in the list 
  $\underline x = (x_1,\ldots,x_n)$ are non-isomorphic. This means that $\gamma$ can be 
  written uniquely as
  $$
  t(\underline a) 
  \isopil (x_{\sigma 1},\ldots,x_{\sigma n}) 
  \isopil (x_1,\ldots,x_n)
  $$
  where the first part is a tuple of maps $(\gamma_1,\ldots,\gamma_n)$ 
  in $X$, with
  $
  \gamma_i : t(a_i) \isopil x_{\sigma i}
  $,
  and the second part is the permutation $\sigma^{-1}$. Different $\sigma$ lie in distinct 
  connected components of $D$
  because the connecting isomorphisms $\theta$ cannot involve permutation.
  It follows that $D$ splits as
  a sum indexed over permutations
  $$
  D = \sum_{\sigma\in \Sigma_n}
  D_\sigma .
  $$
  Inside each summand $D_\sigma$, everything is coordinate-wise:
  the points of $D_\sigma$ are
  pairs of tuples $(\underline a, \underline \gamma)$ 
  where  $\underline a = (a_1,\ldots,a_n)$ and $\underline \gamma = 
  (\gamma_1,\ldots,\gamma_n)$ with
  $\gamma_i : t(a_i) \isopil x_{\sigma i}$, and the morphisms in $D_\sigma$
  are also coordinate-wise, as already noted. It follows that $D_\sigma$ 
  is a product of groupoids
  $$
  D_\sigma = \prod_{i=1}^n D_{\sigma,i} ,
  $$
  where $D_{\sigma,i}$ is the groupoid of pairs $(a_i, \gamma _i)$ with
  $a_i$ in $A$ such that $s(a_i)=x_i$ and $\gamma_i: t(a_i) 
  \isopil x_{\sigma i}$ (and morphisms given by compatible
  $a_i\to a_i'$). This means we have
  $$
  D_{\sigma,i} \simeq {}_{x_i}(A)_{x_{\sigma i}} ,
  $$
  one of the two-sided fibres of the original span.

  We have now arrived at an equivalence of groupoids
  $$
  D \simeq \sum_{\sigma\in \Sigma_n} \prod_{i\in n} {}_{x_i}(A)_{x_{\sigma 
  i}} ,
  $$
  and it remains to check the signs.
  The entries ${}_{x_i}(A)_{x_{\sigma 
  i}}$ may have their own signs, which in the product agrees with the 
  signs of $D_\sigma$, 
  but in the definition of $D$ there are 
  new signs coming from the $2$-cell data in the square defining $Q$. The 
  signs are precisely $\sign(\sigma)$, so we arrive at the Leibniz 
  expansion formula, now stated at the groupoid level as an equivalence of 
  scalars
  $$
  D = \sum_{\sigma\in \Sigma_n} \sign(\sigma) \prod_{i\in n} {}_{x_i}(A)_{x_{\sigma i}} .
  $$
  
  \vspace{-1.2\baselineskip}
\end{proof}

\begin{proof}[Proof of the Corollary.]
  The cardinality of $\Det(\widehat A)$ is
  the cardinality of the two-sided fibre $D = {}_{\underline x} (\Hak n
  A)_{\underline x}$ divided by the symmetry factor $\norm{\underline
  x!}$. By Proposition~\ref{Det=det}, this cardinality expands to
  $$
  \sum_{\sigma\in \Sigma_n} \sign(\sigma) \prod_{i\in n} 
  \norm{{}_{x_i}(A)_{x_{\sigma i}}} \, / \, \norm{\underline x!}
  =
  \sum_{\sigma\in \Sigma_n} \sign(\sigma) \prod_{i\in n} \left( 
  \rule[-2pt]{0pt}{12pt}
  \norm{
  {}_{x_i}(A)_{x_{\sigma i}}}/\norm{x_{\sigma i}!} \right)
  $$
  (using $\underline x!= x_1!\cdots x_n!$ from Lemma~\ref{x!} and 
  permuting the factors), and this expression is precisely the classical 
  (number-level) 
  Leibniz expansion of the determinant of the rational matrix $\norm{\widehat A}$ (whose 
  entries are
  $
  \norm{{}_{x_i} A_{x_j}}/\norm{x_j !}
  $).
\end{proof}

\begin{blanko}{Remark.}
  In applications to combinatorics, the span $X \leftarrow A \to X$ is
  often just a span of finite sets (that is, a directed graph). In this 
  case, for
  any $k\in \N$, the span $\Hak k X \leftarrow \Hak k A \to \Hak k X$ has
  the property that all its two-sided fibres are discrete again. Indeed,
  the functor $\Hak k$ preserves discrete fibrations, so both legs in the
  span are discrete fibrations. In the calculation of the two-sided
  fibres, it follows that $P$ and $Q$ are discrete. Therefore finally the
  groupoid $D$ is discrete (but has signs). Furthermore, as already 
  observed in \ref{discrete-ext}, the automorphism groups of orientable 
  points in $\Hak k X$ are trivial, so the cardinality matrix of (the 
  orientable part of) $\Hak k A$ is integral.
\end{blanko}

\subsection{Immaterial parts of the determinant}
\label{immaterial}

In the determinant calculation above, we were only concerned with 
the `material' part, the one
actually visible in cardinality. As mentioned, the objective definition of 
determinant is much bigger and contains a lot of redundant stuff that 
cancels out. That stuff is there nevertheless. We do not know if it has 
any meaning. Here we just do a very small calculation to get at least a 
glimpse of it. 

\begin{blanko}{Example.}
  Consider the category given by a split idempotent: this category has two 
objects, $x$ and $y$, and five arrows, depicted like this:
\[
\begin{tikzcd}
x \ar[rr, bend left, "s"] && y \ar[ll,  bend left, "p"] 
\ar[loop right, distance=3em, start anchor={[yshift=1ex]east}, end 
anchor={[yshift=-1ex]east},  "e=s\circ p"]{}
\end{tikzcd}
\]
where $p \circ s = \id_x$, and where $s \circ p =: e$ is idempotent. 
We consider the span given by the underlying graph, so
the span in question has $X = \{x,y\}$ and $A = \{x,y,s,p,e\}$, with 
the left leg $d_1$ given by `domain' and the right leg $d_0$ by `codomain'. 
The cardinality matrix is
$\left(\begin{smallmatrix}1&1\\1&2\end{smallmatrix}\right)$
with determinant $1$.

The objective
determinant is given by taking $\Hak 2$ on the whole span. 
We see that $\Hak 2 X$ has four
elements, $xx,xy,yx,yy$ (where we write $xx$ rather than the proper 
$x\wedge x$, just to save space). There are arrows going back and forth between
$xy$ and $yx$, and there are non-identity (and therefore odd) automorphisms of $xx$ and $yy$.
So altogether $\Hak 2 X$ has three connected components, of which only
one is orientable, namely the one with $xy$ and $yx$. 
We know already that in order to
calculate the cardinality of the determinant, it is enough to compute the
$(xy,xy)$ entry of the matrix (the white entry in the following picture), 
but to illustrate the immaterial part, we 
actually calculate all $9$ two-sided fibres of $\Hak 2 A$ (with the 
immaterial entries in grey):

\colorlet{lightgrey}{gray!20}

\begin{center}
  \begin{tikzpicture}

    \node at (-0.3,6.0) {\footnotesize $xy$};	
    \node at (-0.3,3.6) {\footnotesize $xx$};	
    \node at (-0.3,1.2) {\footnotesize $yy$};	

    \node at (1.2,7.5) {\footnotesize $xy$};	
    \node at (3.6,7.5) {\footnotesize $xx$};	
    \node at (6.0,7.5) {\footnotesize $yy$};	

    \begin{scope}[shift={(0.0, 4.8)}]
      \node at (0.65,2.0) {\footnotesize $(xy,\id)$};	
      \node at (0.65,1.45) {\footnotesize $(xe,\id)$};	
      \node at (1.75,0.95) {\footnotesize $(sp,\tau)$};	
    \end{scope}
    \begin{scope}[shift={(0.0, 2.4)}]
      \fill[lightgrey] (0.0,0.0) rectangle ++(2.4,2.4);
      \node at (0.65,2.0) {\footnotesize $(xs,\id)$};	
      \node at (1.75,2.0) {\footnotesize $(sx,\tau)$};	
    \end{scope}
    \begin{scope}[shift={(0.0,0.0)}]
      \fill[lightgrey] (0.0,0.0) rectangle ++(2.4,2.4);
      \node at (0.65,2.0) {\footnotesize $(py,\id)$};	
      \node at (1.75,2.0) {\footnotesize $(yp,\tau)$};	
      \node at (0.65,1.45) {\footnotesize $(pe,\id)$};	
      \node at (1.75,1.45) {\footnotesize $(ep,\tau)$};	
    \end{scope}
    \begin{scope}[shift={(2.4,4.8)}]
      \fill[lightgrey] (0.0,0.0) rectangle ++(2.4,2.4);
      \node at (0.65,2.0) {\footnotesize $(xp,\id)$};	
      \node at (1.75,2.0) {\footnotesize $(xp,\tau)$};	
    \end{scope}
    \begin{scope}[shift={(2.4,2.4)}]
      \fill[lightgrey] (0.0,0.0) rectangle ++(2.4,2.4);
      \node at (0.65,2.0) {\footnotesize $(xx,\id)$};	
      \node at (1.75,2.0) {\footnotesize $(xx,\tau)$};	
    \end{scope}
    \begin{scope}[shift={(2.4,0.0)}]
      \fill[lightgrey] (0.0,0.0) rectangle ++(2.4,2.4);
      \node at (0.65,2.0) {\footnotesize $(pp,\id)$};	
      \node at (1.75,2.0) {\footnotesize $(pp,\tau)$};	
    \end{scope}
    \begin{scope}[shift={(4.8,4.8)}]
      \fill[lightgrey] (0.0,0.0) rectangle ++(2.4,2.4);
      \node at (0.65,2.0) {\footnotesize $(sy,\id)$};	
      \node at (1.75,2.0) {\footnotesize $(sy,\tau)$};	
      \node at (0.65,1.45) {\footnotesize $(se,\id)$};	
      \node at (1.75,1.45) {\footnotesize $(se,\tau)$};	
    \end{scope}
    \begin{scope}[shift={(4.8,2.4)}]
      \fill[lightgrey] (0.0,0.0) rectangle ++(2.4,2.4);
      \node at (0.65,2.0) {\footnotesize $(ss,\id)$};	
      \node at (1.75,2.0) {\footnotesize $(ss,\tau)$};	
    \end{scope}
    \begin{scope}[shift={(4.8,0.0)}]
      \fill[lightgrey] (0.0,0.0) rectangle ++(2.4,2.4);
      \node at (0.65,2.0) {\footnotesize $(yy,\id)$};	
      \node at (1.75,2.0) {\footnotesize $(yy,\tau)$};	
      \node at (0.65,1.45) {\footnotesize $(ey,\id)$};	
      \node at (1.75,1.45) {\footnotesize $(ey,\tau)$};	
      \node at (0.65,0.95) {\footnotesize $(ye,\id)$};	
      \node at (1.75,0.95) {\footnotesize $(ye,\tau)$};	
      \node at (0.65,0.4) {\footnotesize $(ee,\id)$};	
      \node at (1.75,0.4) {\footnotesize $(ee,\tau)$};	
    \end{scope}

    \draw[step=24mm,thin] (0,0) grid (72mm,72mm);

  \end{tikzpicture}
\end{center}

All elements consist of a pair of morphisms with a permutation.
(The pair of morphisms is written as juxtaposition, such as $ee$ in the last element.)
Elements with the identity permutation $\id$ are positive, and 
elements with the transposition $\tau$ are negative. 
Generally, according to the description from the proof of 
Proposition~\ref{Det=det},
the elements in the $ab$-$cd$ entry are 
\begin{quote}
pairs of morphisms 
$\quad a \stackrel f \to c \quad  b \stackrel g \to d \quad$
with identity permutation $\qquad\quad$ written $(fg,\id)$
\phantom{xxxx}
\end{quote}
or
\begin{quote}
pairs of morphisms
$\quad a \stackrel f \to d \quad  b \stackrel g \to c \quad$
with transposition $\tau$ (interchanging the two codomains).
\end{quote}

In the orientable part, the $xy$-$xy$-fibre 
of $\Hak 2 A$, the three possible elements are:

\noindent
$\bullet$ the pair of morphisms $xy$ going from object pair $xy$ to 
object pair $xy$ (hence 
having permutation $\id$), 

\noindent
$\bullet$ the pair $xe$ also going from $xy$ to $xy$ (with $\id$), 

\noindent
$\bullet$ the pair $sp$ going from $xy$ to the transposition of $xy$ 
(and hence has a $\tau$). 

Altogether, the `material part' of
the determinant has cardinality $2-1=1$ (in
  accordance with the cardinality matrix
  $\left(\begin{smallmatrix}1&1\\1&2\end{smallmatrix}\right)$).

  All the grey entries have net cardinality zero (and do not properly form part 
of the cardinality of the matrix). For example, 
in the $xx$-$xy$-entry, the pair of morphisms 
$xs$ goes from $xx$ to $xy$ (with $\id$), whereas $sx$ goes from $xx$ to 
the transposition of $xy$ (and hence has a $\tau$). The 
$yy$-$xy$-entry is similar.
In the $xx$ and $yy$ columns of the matrix,
since the codomain is a symmetric pair, each entry appears both with $\id$ 
and with $\tau$, clearly cancelling out.
\end{blanko}

\begin{blanko}{Historical remark and a challenge.}
  In the discrete case, a span $X\leftarrow A \to X$ is a directed graph. Interpretations of determinants (and the Leibniz
  expansion in particular) in terms of directed graphs go back at least
  to Jackson~\cite{Jackson:1977} and Foata~\cite{Foata:1980}, with many
  results on matrix identities beautifully unified by
  Zeilberger~\cite{Zeilberger:1985}, in terms of graphs weighted by
  commuting indeterminates. Their set-up is symbolic, and the signs are
  introduced somewhat ad hoc, though.
  
  The present set-up is rather a systematic framework at the objective
  level, with the feature that the signs arise completely naturally. We
  believe it should be possible (and pose it as an interesting challenge)
  to redo Zeilberger's proofs naturally in our set-up, since it
  seems that in each case the signs have a permutation origin. We hope
  that the clever cancellations (sign-reversing involutions) in
  \cite{Zeilberger:1985} will turn out to be just non-orientable parts of
  more natural groupoid equivalences.
\end{blanko}


\end{document}